\documentclass[reqno,11pt]{amsart}
\usepackage{amsmath,amssymb}
\usepackage{graphicx}
\usepackage[top=1in,bottom=1in,left=1in,right=1in]{geometry}
\newtheorem{thm}{Theorem}[section]

\newtheorem{lem}[thm]{Lemma}
\newtheorem{prop}[thm]{Proposition}
\newtheorem{defn}{Definition}[section]
\theoremstyle{definition}
\newtheorem{rem}{Remark}[section]
\numberwithin{equation}{section}

\usepackage{color}

\newcommand{\eps}{\varepsilon}

\newcommand{\cL}{\mathcal{L}}
\def\R{\mathbb R}
\def\N{\mathbb N}
\def\be{\begin{equation}}
\def\ee{\end{equation}}
\def\rife#1{(\ref{#1})}

\def\de{\delta}
\def\vfi{\varphi}

\begin{document}
\title[Superquadratic viscous Hamilton-Jacobi equation] 
{Blow-up and regularization  rates, \\ loss and recovery of boundary conditions for \\ the superquadratic viscous Hamilton-Jacobi equation}

%\subjclass[2010]{35K55, 35B40, 35B44, 82C24.}%
%\keywords{Diffusive Hamilton-Jacobi equations, KPZ model, gradient blow-up, isolated boundary singularities, blow-up profile, tangential profile, anisotropic singularity.}

\author[Porretta]{Alessio Porretta}%
\address{Universit\`a di Roma Tor Vergata,
Dipartimento di Matematica,
Via della Ricerca Scientifica 1,
00133 Roma, Italia}
\email{porretta@mat.uniroma2.it}

\author[Souplet]{Philippe Souplet}%
\address{Universit\'e Paris 13, Sorbonne Paris Cit\'e,
CNRS UMR 7539, Laboratoire Analyse, G\'{e}om\'{e}trie et Applications,
93430 Villetaneuse, France}
\email{souplet@math.univ-paris13.fr}

\begin{abstract}
 We study the qualitative properties of the unique global viscosity solution of the superquadratic diffusive Hamilton-Jacobi equation
with (generalized) homogeneous Dirichlet conditions.
We are interested in the phenomena of gradient blow-up (GBU), loss of boundary conditions (LBC), recovery of boundary conditions and eventual regularization, and in their mutual connections.

In any space dimension, we establish the sharp minimal rate of GBU.
Only partial results were previously known except in one space dimension.
We also obtain the corresponding minimal regularization rate.

In one space dimension, under suitable conditions on the initial data, we give a quite detailed description of the  behavior of solutions for all $t>0$. In particular, we show that
nonminimal  
GBU solutions immediately lose the boundary conditions after  
 the blow-up time and are immediately regularized after recovering the boundary data. Moreover, both GBU and regularization occur with the minimal rates, while loss and recovery of boundary data occur with linear rate. We describe further the intermediate singular life of those solutions in the time interval between GBU and  regularization. 

We also study minimal GBU solutions, for which GBU occurs {\it without} LBC: those solutions are immediately regularized, but their GBU and regularization rates are more singular.

Most of our one-dimensional results crucially depend on zero-number arguments,
which do not seem to have been used so far in the context of viscosity solutions of Hamilton-Jacobi equations.
\end{abstract}

\maketitle

%\tableofcontents

% ----------------------------------------------------------------

\section{Introduction}

Let $\Omega$ be a smooth  bounded domain of $\R^n$. In this article we consider the diffusive Hamilton-Jacobi equation (also called viscous Hamilton-Jacobi equation)
\be\label{VHJ}
\begin{cases}
u_t -\Delta u = |\nabla u|^p, & t>0,\ x\in \Omega, \\
u=0,  & t>0,\ x\in \partial\Omega, \\
u(0,x)= \phi, & x\in \Omega,
\end{cases} 
\ee
with $p>2$ and $\phi\in X$, where
 $$
X=\bigl\{\phi\in C^1(\overline\Omega),\ 
 \phi=0 \hbox{ on $\partial\Omega$}\bigr\}. 
 $$
Nowadays, it is well known that problem \rife{VHJ}  exhibits very interesting phenomena whenever $p>2$, i.e.~in the case of  superquadratic growth of the nonlinearity. 
  By standard theory, problem \rife{VHJ} admits a unique, maximal classical solution with existence time $T^*(\phi)\in(0,\infty]$.  
Even if the initial data $\phi$ is smooth and satisfies the compatibility condition $\phi=0$ at the boundary $\partial \Omega$, 
the classical solution  of  \rife{VHJ} may blow up 
 in finite time
  i.e., $T^*(\phi)<\infty$, in which case 
$$\lim_{t\to T^*_-}\|u(t)\|_{C^1(\overline\Omega)}=\infty.$$
This does happen for suitably large $\phi$. 
On the other hand, by maximum principle, $u$ cannot blow up in $L^\infty$-norm; so what happens is that $\|\nabla u(t)\|_\infty$ 
has to blow up while the solution itself  remains bounded. This is referred to in the literature  as {\it gradient blow-up} (GBU). Several facts are by now well understood, and two main points are to be recalled:

\smallskip

\begin{itemize}
 \item[(a)] GBU can only happen at the boundary $\partial \Omega$ (as a consequence of local interior universal gradient estimates), 
 see e.g. \cite{SZ06}. 

\smallskip
 \item[(b)] The solution survives after the blow-up time and  it can be continued as a generalized viscosity solution of \rife{VHJ}. To be more precise, problem \rife{VHJ} admits a unique (generalized) viscosity solution $u$ in $(0,\infty) \times \Omega$; this solution $u$ is continuous in $[0,\infty)\times \overline \Omega$,    with $u\ge 0$ on $[0,\infty)\times\partial\Omega$,   and coincides with the (unique) classical solution in $(0,T^*)$
(see \cite{BDL04}).

\end{itemize}

\smallskip
But there is more than that. When we say a {\it generalized} viscosity solution, this means that the boundary condition should be understood in a relaxed sense, as it is done for example in first order problems. Namely, the prescribed boundary condition may not be satisfied by the solution in the classical sense, but rather in a weaker sense which is encoded in the viscosity formulation. 
Somehow, when the gradient blows-up at the boundary, one may need to relax the boundary prescription in order that the problem be still solvable; 
the superquadratic 
 growth of the nonlinearity overcomes the diffusive 
smoothing of the second order term and a  {\it first order behavior} is observed.
\smallskip

Recalling that $u\in C([0,\infty) \times \overline \Omega)$, we see that such solutions,
which are meant to satisfy zero boundary conditions in a generalized sense,
nevertheless have to continuously take on {\it positive} boundary values 
(at some points) for some times after the blow-up time.
This apparently paradoxical situation can however be interpreted in a more intuitive way,
when one recalls that the global viscosity solution can also be obtained as the limit 
of a sequence of global classical solutions of regularized versions of problem \rife{VHJ}, with truncated nonlinearity
(see Section~\ref{Sec3}). Since this convergence is monotone increasing but not uniform up to the boundary,
the loss of boundary conditions can in this framework be seen as a more familar boundary layer phenomenon.

\smallskip
The possibility of {\it loss of (classical) boundary conditions} (LBC) for problem \rife{VHJ} was suggested in \cite{BDL04}, 
and its actual occurence was recently confirmed in \cite{PS2}, \cite{QR16} for suitable initial data:

\smallskip

\begin{itemize}
 \item[(c)] 
For all sufficiently large initial data $\phi$, the solution $u$ undergoes GBU as well as LBC.
Namely:
$$\hbox{$T^*(\phi)<\infty$ and there exist $t_0>T^*(\phi)$ and $x_0\in\partial\Omega$ such that $u(t_0,x_0)>0$.}$$

\end{itemize}

\noindent But another key property of solutions, which was known before (see \cite{PZ}), is the eventual {\it recovery of boundary conditions:}
\smallskip

\begin{itemize}
 \item[(d)] Any GBU solution eventually becomes a classical solution (including the boundary conditions) for all $t$ large enough.
Moreover,  any solution  $u(t)$ decays exponentially to $0$ in $C^1(\overline\Omega)$ as $t\to\infty$.
\end{itemize}

\smallskip

Therefore, the behavior of solutions of \rife{VHJ} presents   three main issues: 

\vskip 1pt

\begin{itemize}
 \item[(i)] gradient blow-up

 \item[(ii)] loss of boundary condition

 \item[(iii)] recovery of boundary condition and regularization. 

\end{itemize}

\vskip 1pt

 \noindent 

Up to now, no description of the way the solution loses and/or recovers its boundary conditions seems to be known.
As for the GBU phenomenon itself, it has been investigated so far in several papers and many things are known,
although several questions remain open. 
Sufficient blow-up conditions appear in \cite{ABG89}, \cite{A96}, \cite{S02}, \cite{HM04}. 
It is known that GBU is localized on some part of the boundary \cite{SZ06} and that single-point  
blow-up may occur \cite{LS}, \cite{Est18}. 
The     space profiles at $t=T^*$  are also investigated (see \cite{CG96}, \cite{ARS04}, \cite{SZ06}, \cite{GH08}, \cite{PS}).
Partial results on the GBU rate are available (see \cite{CG96}, \cite{GH08}, \cite{ZL13}).
   On the other hand, for related one-dimensional equations, typically of the form $u_t-u_{xx}=e^{u_x}$, with zero boundary conditions,
a global weak continuation after GBU was constructed in \cite{FL94}.
As an important difference with \rife{VHJ}, the boundary conditions of this weak continuation remain lost for all $t>T^*$.

\smallskip

The purpose of this article is to give new results on all aspects (i)-(iii), investigating the connections between the three questions. 

\smallskip

In any space dimension, we establish the sharp minimal rate of GBU.
Only partial results were previously known (\cite{GH08}, \cite{ZL13}), except in one space dimension \cite{CG96}.
We also obtain the corresponding minimal regularization rate, which had not been studied before.

\smallskip

In one space dimension, under suitable conditions on the initial data, we give a quite detailed description 
of the qualitative behavior of solutions for all $t>0$.
Most of our one-dimensional results crucially depend on zero-number arguments,    applied to the function $u_t$. 
Zero-number arguments   do not seem to have been used so far in the context of 
viscosity solutions of Hamilton-Jacobi equations.

\smallskip

For {\it nonminimal} GBU solutions (see Definition~\ref{defmin}), we especially describe the intermediate singular life of the solution 
in the time interval between GBU and eventual regularization, showing:

\vskip 1pt

\begin{itemize}
 \item[-] immediate LBC after GBU and immediate regularization after recovery of boundary conditions;

 \item[-]  loss and recovery of boundary conditions occur at linear rates;

 \item[-] $C^1$ regularity of the boundary value as a function of time in the interval of LBC; 
 
 \item[-]  GBU and regularization according to the minimal rates.

 \end{itemize}

\vskip 1pt

\noindent We note that upper estimates on the GBU rate were previously known only in the case of inhomogeneous boundary conditions or forcing terms
 (for some class of one-dimensional or radial solutions, see \cite{GH08}, \cite{QS07}, \cite{ZL13})),
due to time-monotonicity restrictions of the existing methods.

\smallskip

On the other hand, we recently showed in \cite{PS2} that    (in any space dimension) the LBC set can be quite arbitrary depending on the initial data, 
and that gradient blow-up may also occur {\it without} LBC, typically for minimal GBU solutions.  
Here we prove that, in one space dimension, minimal GBU solutions have a completely different behavior 
than nonminimal  
 GBU solutions: they do not lose boundary data,  and 
they are immediately regularized, 
but their GBU and regularization rates are more singular.

\section{Main results}

We split them in three subsections.
The first one concerns minimal GBU and regularization rates and covers the general $n$-dimensional problem.
In the last two we specialize to the one-dimensional problem and  we  give detailed description of a class of GBU solutions, respectively with or without LBC.

\subsection{Minimal GBU and regularization rate}
We start with the general, minimal lower estimate of the GBU rate.

\begin{thm}\label{minimal_rate}
Let $\phi\in X$ and assume $T^*=T^*(\phi)<\infty$. 
Then there exists a constant $C_1>0$ such that
\be\label{lower-bu0}
\|\nabla u(t)\|_\infty\ge C_1(T^*-t)^{-1/(p-2)},\quad t\to T^*_-.
\ee
\end{thm}

Theorem~\ref{minimal_rate} improves the result obtained in \cite{GH08},
where the lower blowup estimate was only obtained under the
weaker form 
\be\label{lower-bu-weaker}
\sup_{s\in[0,t]}\|\nabla u(s)\|_\infty\ge C(T^*-t)^{-1/(p-2)}. 
\ee
Estimate \rife{lower-bu0} had been proved earlier in the case $n=1$ by a completely different method
(intersection-comparison, which does not extend to higher dimensions).
  We here refine some ideas from \cite{GH08} and \cite{QS07}.
The improvement from \rife{lower-bu-weaker} to \rife{lower-bu0} stems from a rather delicate argument based on the variation of constants formula.

\medskip

  Next, in view of property (b) in the previous subsection, we may introduce
the ultimate\footnote{We cannot a priori exclude that the solution loses and recovers regularity and/or boundary conditions again at several times, although we 
presently do not know examples of such solutions.}
regularization time $T^r(\phi)$, defined by:
\be\label{defTr}
T^r(\phi):=  \inf  \bigl\{\tau >T^*(\phi);\ u(t,\cdot)\in C^1_0(\overline\Omega) \hbox{ for all } t>\tau\bigr\} \in [T^*(\phi),\infty).
\ee
Our next main result is a minimal lower estimate for the regularization rate.

\begin{thm}\label{minimal_rate2}
Let $\phi\in X$ with $T^*=T^*(\phi)<\infty$ and let $T^r=T^r(\phi)$ be defined by \rife{defTr}.
Then there exists a constant $C_2>0$ such that
$$
\|\nabla u(t)\|_\infty\ge C_2(t-T^r)^{-1/(p-2)},\quad t\to T^r_+.
$$
\end{thm}

Theorem~\ref{minimal_rate2} seems completely new. We refer to Section~\ref{sec-rates} for further results related with 
Theorems \ref{minimal_rate} and \ref{minimal_rate2}.

\medskip

\subsection{Description of a class of one-dimensional GBU solutions with LBC}

\smallskip

An important notion in the subsequent analysis is that of minimal/non-minimal GBU solution.
 For shortness, here we use the notation $T^*(u)$ and $T^*(v)$ rather than $T^*(u(0))$, $T^*(v(0))$, respectively.

\begin{defn} \label{defmin}
Let $n\ge 1$. A solution $u$ of \rife{VHJ} is called a minimal GBU solution if  
 $T^*(u)<\infty$ and every solution $v$ such that $v(0)\leq u(0)$ and $v(0) \neq  u(0)$ is  a global classical solution, i.e. $T^*(v)=\infty$.
\end{defn}

This definition is consistent with the monotonicity of the function $T^*$. Indeed, if $v_0\le u_0$ then $T^*(v_0)\ge T^*(u_0)$ 
(see, e.g.,  \cite{LS}); and if, moreover, $v$ undergoes LBC, then so does $u$ -- see \cite{PS2}).
The existence of minimal blow-up solutions in any dimension was shown in  \cite{PS2}, where such solutions were constructed as threshold solutions.
Namely for any $\psi\in X$ with $\psi\ge 0$ and $\psi\not\equiv 0$, 
one considers $\lambda^*=\sup\{\lambda>0;\,T^*(\lambda\psi)=\infty\}$ and
shows that $\lambda^*\in (0,\infty)$ and that $T^*(\lambda^*\psi)<\infty$ (the latter fact is not obvious).

\smallskip

  We also recall (see \cite[Theorem 3]{PS2}) that,  at least in one space dimension, loss of boundary conditions 
occurs if and only if the GBU solution is nonminimal.\footnote{Actually   it is not difficult to show that the implication  
'LBC $\Rightarrow$ $u$ nonminimal' remains valid in any dimension, but we are presently unable to show the reciprocal
  in general.  }

\begin{prop}\label{prop-min} 
Let $\Omega=(0,1)$. Then $u$ is a minimal blow-up solution if and only if there is no loss of boundary conditions.
\end{prop}

 It turns out that, for a suitable class of initial data on $\Omega=(0,1)$, we are able to provide a fairly precise description of the 
life of nonminimal solutions, including their GBU, LBC and regularization behaviors in time and in space.
We consider the initial data   $\phi\in W^{3,\infty}(0,1)$  
satisfying the following properties:
\be\label{hypID1}
\hbox{$\phi$ is symmetric w.r.t. $x=\frac12$,\ \ $\phi'\ge 0$ on $[0,\frac12]$,\ \  $\phi(0)=\phi''(0)+{\phi'}^p(0)=0$,}
\ee
\be\label{hypID2}
\hbox{there exists $a\in (0,1/2)$ such that } 
\phi''+{\phi'}^p
\begin{cases}
\,\ge 0 & \hbox{on $[0,a]$} \\
\,\le 0 & \hbox{on $[a,\frac12]$.} \\
\end{cases}
\ee
These assumptions are motivated by intersection-comparison -- or zero-number -- arguments that will be crucially used  in our proofs.
Minimal and nonminimal GBU solutions starting from such initial data can be easily constructed   (see Lemma~\ref{constuctID}).
Also we will denote by $U_*$ the one-dimensional stationary solution 
\be\label{defUstar}
U_*(x):= c_p\, x^{\frac{p-2}{p-1}}, \quad x\ge 0, \qquad\hbox{with }  c_p= (p-2)^{-1}(p-1)^{\frac{p-2}{p-1}}.
\ee
Note that $U_*$ is singular near $x=0$ (since $U_*'(x)=((p-1)x)^{-1/(p-1)}$) and recall, as is well known,  that $U_*$ plays the role of a reference GBU profile
(see, e.g., \cite{CG96}, \cite{ARS04}).

\begin{thm}\label{nonmin0} 
Let $\Omega=(0,1)$, let    $\phi\in W^{3,\infty}(0,1)$  satisfy \eqref{hypID1}-\eqref{hypID2},
and assume that $u$ is a nonminimal GBU solution of \rife{VHJ}.
\smallskip

\begin{itemize}
 \item[(i)] There is immediate loss of boundary conditions after GBU
and immediate and permanent regularization after recovery of boundary conditions. 
More precisely, we have
$$u(t,0)>0,\quad\hbox{for all $t\in(T^*,T^r),$}$$
and $u$ is a classical solution (including boundary conditions) for all $t\in(T^r,\infty)$
(recalling the definition in \rife{defTr}).

\smallskip
 \item[(ii)]   $u$ satisfies the following GBU, detachment, reconnection and regularization estimates:
\be\label{estA0}
c_1(T^*-t)^{-1/(p-2)}\le \|u_x(t)\|_\infty \le c_2(T^*-t)^{-1/(p-2)},\quad\hbox{ as $t\to T^*_-,$}
\ee
\be\label{estB0}
u(t,0) \sim \ell_1(t-T^*),\quad\hbox{ as $t\to T^*_+,$}
\ee
\be\label{BeqConclAa0}
u(t,0) \sim \ell_2(T^r-t),\quad\hbox{ as $t\to T^r_-,$}
\ee
\be\label{BeqConclA0}
c_3(t-T^r)^{-1/(p-2)} \le \|u_x(t)\|_\infty\le c_4(t-T^r)^{-1/(p-2)},\quad\hbox{ as $t\to T^r_+,$}
\ee
with some constants $c_i,\ell_i>0$.

\smallskip
 \item[(iii)]  In the interval $[T^*,T^r]$, the solution behaves near the boundary like a shifted copy of the singular stationary profile $U_*$:
  there exists $K>0$ such that
$$|u(t,x)-u(t,0)-U_*(x)|\le K\,  {\frac{  x^2} 2} \quad\hbox{ in $[T^*,T^r]\times [0,1/2]$},$$
$$|u_x(t,x)-U_*'(x)|\le Kx\quad\hbox{and}\quad  |u_{xx}(t,x)-U_*''(x)|\le K \ \quad\hbox{ in $[T^*,T^r]\times (0,1/2]$},$$ 
and the restriction of the function $u-U_*$ to $(T^*,T^r]\times [0,1/2]$ is of class $C^1$ in $t$ and $x$.

\vskip2pt

\noindent Furthermore, the boundary value enjoys the following monotonicity and regularity properties:
\vskip2pt

 \begin{itemize}
 
 \item[$\bullet$]  $u(t,0)$ admits a unique global maximum at some $T_m\in (T^*,T^r)$;

\vskip1pt

 \item[$\bullet$] $u(t,0)$ is  increasing on $[T^*,T_m]$ and  decreasing on $[T_m,T^r]$;

\vskip1pt

 \item[$\bullet$]  the restriction of the function $t\mapsto u(t,0)$ to $[T^*,T^r]$ is of class $C^1$
and $u_t(t,0)$ undergoes jump discontinuities at $t=T^*$ and $t=T^r$. 

 \end{itemize}
  
\vskip2pt

\noindent In addition, we have
$$
u_t\le 0\quad\hbox{ in $[T_m,\infty)\times (0,1)$.}
$$
   \end{itemize}
\end{thm}

The shape of the solution described in Theorem~\ref{nonmin0} is depicted in the following figure.

$$
\includegraphics[width=12cm, height=5.5cm]{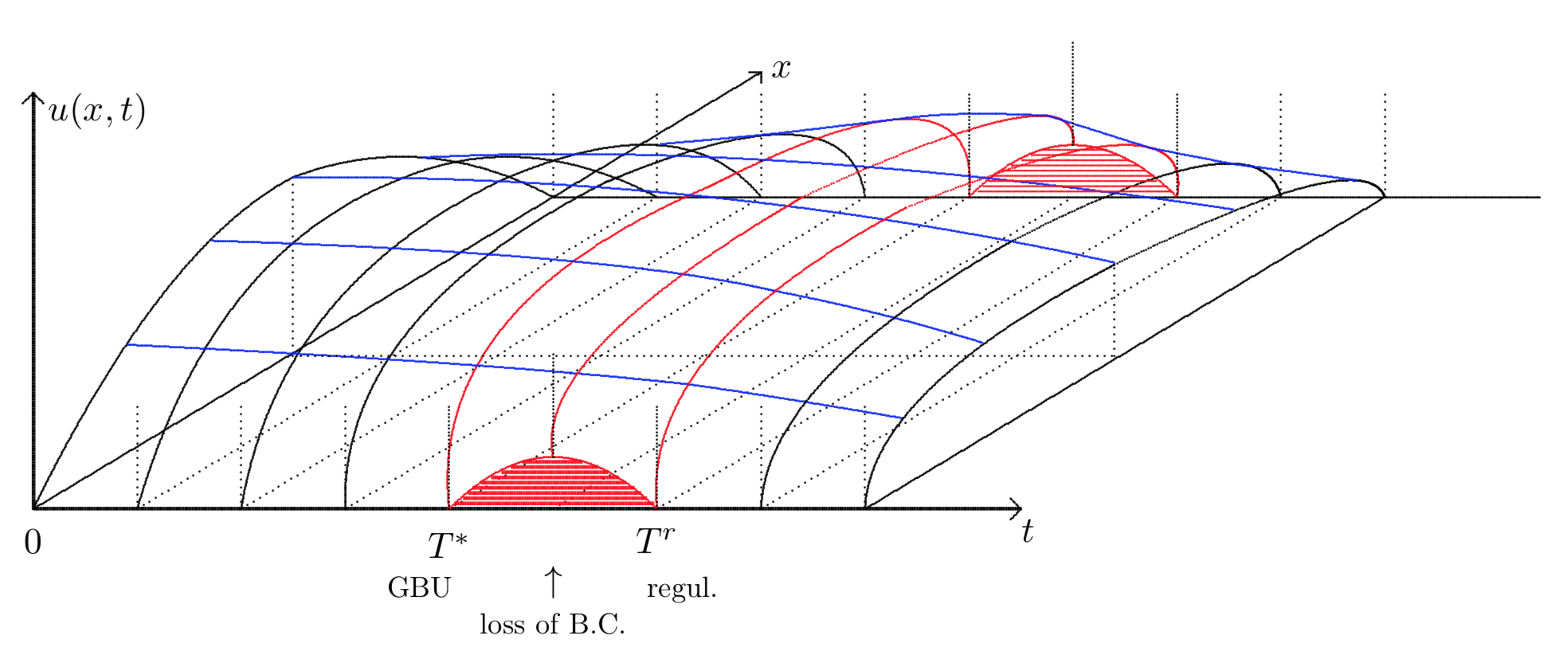}
$$

More generally, as shown in our next result,   
the conclusions of Theorem~\ref{nonmin0} concerning the regularity and behavior of the function $u-U_*$ 
are actually valid for any LBC solution in one space dimension,
on any time interval where the boundary conditions are lost at $x=0$.

In the following, we set
$$
X_1:=\{\phi\in C^1([0,1]),\ \phi(0)=\phi(1)=0\}.
$$

\begin{thm}
\label{proppersist}
Let $\Omega=(0,1)$, and let $\phi\in X_1$ with $T^*(\phi)<\infty$.
Let $T_1, T_2$ be such that $T^*\le T_1<T_2\le T^r$ and 
$$u(t,0)>0 \quad\hbox{ on  $(T_1,T_2)$.}$$

\begin{itemize}

\item[(i)] Then there exists a constant $K>0$ (depending only on $u$), such that
\be\label{shiftedcopy0}
|u(t,x)-u(t,0)-U_*(x)|\le K\, {\frac{x^2}2}\quad\hbox{ in $[T_1,T_2]\times (0,1/2]$},
\ee
\be\label{shiftedcopy}
|u_x(t,x)-U_*'(x)|\le Kx\quad\hbox{and}\quad |u_{xx}(t,x)-U_*''(x)|\le K \ \quad\hbox{ in $[T_1,T_2]\times (0,1/2]$}. 
\ee
\vskip 3pt
\item[(ii)] The restriction of the function $u-U_*$ to $(T_1,T_2]\times [0,1/2]$ is of class $C^1$ in $t$ and $x$
and the restriction of the function $t\mapsto u(t,0)$ to $[T_1,T_2]$ is of class $C^1$.
\end{itemize}
\end{thm}

\begin{rem} 
(a) We note that the solutions in Theorem~\ref{nonmin0} behave according to the minimal GBU and regularization rates 
(compare with Theorems \ref{minimal_rate} and \ref{minimal_rate2}).
As for the linear rates of loss and recovery of boundary conditions, interestingly, we see on the contrary that they are maximal.
Indeed, since all solutions satisfy bounds of the form $|u_t|\le C(t_0)$ in $[t_0,\infty)\times\Omega$ for all $t_0>0$ (see Section~\ref{Sec3}), 
it is immediate that loss or recovery of the zero boundary conditions can never occur at a rate faster than linear.
\smallskip

(b) Theorem~\ref{nonmin0} is the first known result on upper GBU estimates for the homogeneous problem~\rife{VHJ}.
Previously, the upper GBU estimate in \eqref{estA0} was only known for the modified inhomogeneous problem
\be\label{vhjModif}
\begin{cases}
u_t-u_{xx} =|u_x|^p+\lambda, & \quad t>0,\ 0<x<1,\\
u(t,0) =0, \ u(t,1)=M, & \quad t>0,\\
u(0,x) =\phi(x),  & \quad 0<x<1,\, 
\end{cases}
\ee
where either $\lambda=0$ and $M>c_p$, or $M=0$ and $\lambda>0$ large enough.
The result in the case $\lambda=0$ and $M>c_p$ was proved in \cite{GH08}
 for time-increasing GBU solutions (such solutions exist only for $M>c_p$).
The argument of \cite{GH08} was modified in \cite{QS07} to cover the case $M=0$ with $\lambda>0$ large enough, still for time-increasing GBU solutions.
The upper GBU rate remains an open problem for the nonradial higher dimensional case
(see \cite{ZL13} for a result on a radial inhomogeneous problem).
\smallskip

(c) 
We see from Theorem~\ref{nonmin0} that $u_t$ exists and is continuous on $((0,\infty)\times [0,1/2])\setminus \{(T^*,0),(T^r,0)\}$.
Moreover, $u_t$ is respectively uniformly positive or negative near the corners.
Namely, under the assumptions of Theorem~\ref{nonmin0}, we have
$$
u_t \ge c \ \hbox{ in $(T^*,t_0)\times(0,a)$}
\quad\hbox{ and }\quad
u_t\le -c\ \hbox{ in $(t_1,T^r)\times (0,1)$}
$$
for some times $t_0, t_1$ with $T^*<t_0<t_1<T^r$, and some $a\in (0,1/2)$ and $c>0$
(see Propositions~\ref{nonmin} and \ref{recon}).
  On the other hand, the restriction of the function $V:=u-U_*$ to $Q:=[T^*,T^r]\times [0,1/2]$ is actually differentiable
(including at $(T^*,0)$), with $V_x$ continuous (see the end of the proof of Theorem~\ref{proppersist}(ii)).
But we do not know if the restriction of $u_t=V_t$ to $Q$ is continuous  
at $(T^*,0)$. 
Some additional qualitative properties of $u$ under the assumptions of Theorems~\ref{nonmin0}, \ref{proppersist} or \ref{minimal0}
are given in Lemmas~\ref{bounduxt}, \ref{boundutt} and \ref{minlem1}--\ref{minlem3} and in Proposition~\ref{zeronumbermonotone}.

\smallskip

(d) Theorem~\ref{nonmin0}, as well as Theorem~\ref{minimal0} below, remain true for a larger class of symmetric initial data, 
characterized by the zero-number of $u_t$ being at most four
(see Remark~\ref{remhigherN}).
The case of higher zero-number remains an open problem. But we suspect that different behaviors might occur,
such as multiple losses and recoveries of boundary conditions.
Also, other interesting and largely open problems concern the behavior of LBC solutions in higher dimensions. 

On the other hand, in Theorems~\ref{nonmin0} and \ref{minimal0}, we have assumed for simplicity that $\phi$ is symmetric and nondecreasing on $[0,1/2]$.
Without these assumptions, some results of the same type might still be obtainable by our methods, 
at the expense of additional technical difficulties.
Since our purpose is to focus on the presentention of new phenomena,
we have refrained from considering  
 such generality.
\end{rem}

\subsection{Description of a class of one-dimensional GBU solutions without LBC}

We already know that minimal solutions have no LBC.
Within the same class of initial data as in the previous subsection, 
we have the following more precise information on the behavior of minimal solutions. 
As a crucial difference, we find in particular that such solutions have more singular blow-up and regularization rates
than the 'standard' ones obtained in Theorem~\ref{nonmin0} for nonminimal solutions.
Moreover, $T^r=T^*$ so that $T^*$ is the only singular time of the solution. 

\begin{thm}\label{minimal0}
Let $\Omega=(0,1)$, let $\phi\in W^{3,\infty}(0,1)$ satisfy \eqref{hypID1}-\eqref{hypID2},
and assume that $u$ is a minimal GBU solution of \rife{VHJ}.
Then:

\begin{itemize}

\item[(i)] The GBU  
rate is more singular than the minimal  
rate:
$$
(T^*-t)^{1/(p-2)}\|u_x(t)\|_\infty\to\infty,\quad t\to T^*_-\,;
$$
\vskip0.3em
\item[(ii)]  There is  instantaneous and permanent regularization at $T^*_+$, namely,
$u$ is a classical solution (including boundary conditions)
on $(T^*,\infty)$;
\vskip0.3em
\item[(iii)] The regularization rate is more singular than the minimal  
 rate:
$$
(t-T^*)^{1/(p-2)}\|u_x(t)\|_\infty\to\infty,\quad t\to T^*_+.
$$
\end{itemize}
\end{thm}
\medskip

\begin{rem} 
The solutions described in Theorem~\ref{minimal0}  are a kind of analogue of the {\it peaking solutions}
known for the semilinear heat equation $u_t-\Delta u=u^p$, $t>0$, $x\in \R^n$.
Such solutions (which exist for $n\ge 3$ in suitable ranges of $p$) blow up (in $L^\infty$ norm) at the origin
as $t\to T_-$ and are instantaneously and forever regularized for $t>T$.
See, e.g., \cite{GV97}, \cite{FMP05}, \cite{MM09} and the references therein for more details.
\end{rem} 
\medskip

The outline of the rest of the paper is as follows. 
A number of preliminary facts are collected in Section~3.
Section~\ref{sec-rates} is devoted to the lower estimates of blow-up and regularization rates, including
the proofs of Theorems \ref{minimal_rate} and \ref{minimal_rate2}.
The rest of the paper is then concerned with the one-dimensional problem.
The useful, preliminary one-dimensional results are given in Section~\ref{prelim1d},
except for the basic material on zero-number, which is the object of Section~\ref{SecZ}.
The behavior of nonminimal solutions is studied in Section~\ref{Sec-nonmin},
which consists of three subsections. The first two are respectively devoted to the GBU and LBC behavior near $T^*$,
and to the reconnection and regularization behavior near~$T^r$.
The last subsection studies the singular life of the solution in the time interval $(T^*,T^r)$.
In particular, the proofs of  
  Theorems~\ref{nonmin0} and \ref{proppersist}  are completed in Subsection~\ref{Sec-nonmin3}.
The behavior of minimal solutions is then studied in Section~\ref{Sec-min}, where Theorem~\ref{minimal0} is proved.
   Finally, the paper is complemented with 
two appendices, respectively devoted to an additional zero-number property and to the proof of 
 a useful approximation result, via Bernstein-type estimates. 

\section{Preliminary facts on the viscous HJ equation}\label{Sec3}

In this section, we recall some    preliminary  facts about the unique solution of the problem 
\begin{eqnarray}
   &u_t-\Delta u=|\nabla u|^p,\ \ \ &t>0,\ x\in \Omega,\label{vhj-1}\\
   &u(t,x)=0,\ &t>0,\ x\in \partial \Omega,\label{vhj-2}\\
   &u(0,x)=\phi(x),\ &x\in \Omega. \label{vhj-3}
\end{eqnarray}
Recall the notation 
 $$
 X=\{\phi\in C^1(\overline\Omega),\ \phi=0 \hbox{ on $\partial\Omega$}\}.
 $$
The following theorem collects a few results established in the literature
 (see, e.g., \cite{BDL04}, \cite[Section 40]{QS07} and the references therein).
 
\medskip

    \goodbreak 
 
 \begin{thm}\label{prelimprop}
Let $\Omega$ be a smooth bounded domain of $\R^n$, $p>2$ and $\phi\in X$. Then:  
 \begin{itemize}
 
 \item[(i)] There exists a unique generalized viscosity solution $u$ of problem \rife{vhj-1}-\rife{vhj-3} for $t\in (0,\infty)$, and  $u\in C([0,\infty) \times \overline \Omega)$   with $u\ge 0$ on $[0,\infty)\times\partial\Omega$. Moreover, if $u,v$ are the solutions corresponding to   different initial data $u_0,v_0\in X$, then 
\be\label{contdep}
\|u(t)-v(t)\|_\infty\le \|u_0-v_0\|_\infty, \quad\hbox{ for all $t>0$.}
\ee

\item[(ii)] There exists $T^*\in (0,\infty]$ such that    $u\in C^{1,2}((0,T^*)\times\overline\Omega)$,
$\nabla u\in C([0,T^*)\times\overline\Omega)$,  $u$ is a classical solution in $(0, T^*)$ and, if $T^*<\infty$, then
$$
\lim\limits_{t \to T^*-} \|\nabla u(t)\|_\infty = \infty\,.
$$

\item[(iii)]  We have $u\in  C^{1,2}((0,\infty) \times \Omega)$, i.e. $u$ is smooth inside the domain~$\Omega$.

\item[(iv)] For all $t>0$ we have $u_t(t,\cdot)\in L^\infty(\Omega)$ 
and, given any $t_0>0$, there exists a constant 
  $M(t_0)>0$  such that
\be\label{u_t-bounded}
\|u_t(t)\|_\infty\le M(t_0),\quad t\ge t_0.
\ee
The same estimate also holds for $t_0=0$ provided $\Delta u_0 + |\nabla u_0|^p \in C(\overline \Omega)$.
  \end{itemize}
  \end{thm}

As mentioned before, it is quite useful to recover the solution $u$ as the limit
of global classical solutions of regularized problems of the form:  
\be\label{approxpbm}
\begin{cases}
\partial_tu_k-\Delta u_k =F_k(\nabla u_k), & \quad t>0,\ x\in\Omega, \\
u_k=0, & \quad t>0,\ x\in\partial\Omega,\\
u_k(0,x)=\phi(x), & \quad x\in\Omega.
\end{cases}
\ee
We note that the simple truncated nonlinearities $F_k(\xi)=\min\bigl(|\xi|^p,k^{p-2}|\xi|^2\bigr)$ are often used.
 However these $F_k$ are only locally Lipschitz continuous and choices with better regularity, and additional features such as convexity,
 may be useful for certain properties
 (see Section~\ref{prelim1d}). 
 
To this end, we formulate the following general approximation result. The first assertion shows
--~without resorting to viscosity solution theory~-- that the maximal classical solution of problem~\rife{VHJ} 
admits a unique global weak continuation $U$
after $t=T^*$, defined in a natural way through monotone approximation.
This notion of ``limit solution'' $U$ is similar to the concept of ``proper extension'', 
which is customary in the blow-up theory for 
nonlinear parabolic equations with zero-order nonlinearities, such as 
\be\label{NLH}
u_t-\Delta u= u^p
\ee
(see \cite{BC87}, \cite{GV97}, \cite{QS07})\footnote{   As an important difference 
with the diffusive Hamilton-Jacobi equation, 
we recall that for equation \rife{NLH} the existence of a global weak continuation after blow-up is an exceptional phenomenon 
(it only occurs for ``minimal'' blow-up solutions in certain ranges of $p$), and that blow-up is generically complete. }.
As for the second assertion, it shows that the limit solution $U$ coincides with the global viscosity solution $u$.

\begin{thm}\label{prelimprop2}
Let $\Omega$ be a smooth bounded domain of $\R^n$, $p>2$ and $\phi\in X$.
\smallskip

(i) Consider a sequence of nonlinearities $F_k(\xi)\in W^{1,\infty}_{loc}(\R^n)$ such that:
\be\label{approxHyp1}
\hbox{$F_k(\xi)$ is nondecreasing with respect to $k$, with
$\lim_k F_k(\xi)=|\xi|^p$ for all $\xi\in\R^n$,}
\ee
\be\label{approxHyp3}
|\nabla F_k(\xi)|\le C_1\bigl(1+|\xi|^{-2}F_k^{2-\theta}(\xi)\bigr)\quad\hbox{for all $\xi\ne 0$,}
\ee
\be\label{approxHyp4}
F_k(\xi)\ge C_2|\xi|^2\quad\hbox{ for all $|\xi|\ge 1$,}
\ee
\be\label{approxHyp4b}
0\le F_k(\xi)\le C(k)(1+|\xi|^2)\quad\hbox{ for all $\xi\in\R^n$,}
\ee
where $\theta\in (0,1), C_1, C_2>0$ 
are constants independent of $k$ and $C(k)>0$ is a constant depending on $k$.
Then problem \rife{approxpbm} has a unique global classical solution $u_k$ and there exists a function
 $U\in C([0,\infty)\times\overline\Omega)\cap C^{1,2}((0,\infty)\times\Omega)$,
with the following properties:
\be\label{approxpbm2}
\hbox{ $u_k$ is nondecreasing with respect to $k$, \quad $\displaystyle\lim_{k\to\infty}u_k=U$   in $C^{1,2}_{loc}((0,\infty)\times\Omega)$,}
\ee
and $U$ solves the problem
\be\label{VHJ2}
\begin{cases}
U_t -\Delta U = |\nabla U|^p, & t>0,\ x\in \Omega, \\
U(0,x)= \phi, & x\in \Omega.
\end{cases} 
\ee
Moreover we have
 \be  \label{BernsteinPhAux6ca}
U\ge 0 \ \hbox{ on $(0,\infty)\times\partial\Omega$}
\ee
and $U$ is independent of the choice of the $F_k$ satisfying the above assumptions.
\smallskip

(ii) The solution $U$ from assertion (i) concides with the global viscosity solution $u$ given by Theorem~\ref{prelimprop}. 
\end{thm}
 
   Results similar to Theorem~\ref{prelimprop2} are probably known but we could not find in the literature a statement and proof suitable to our needs.
 We therefore give in Appendix a proof, based on a simple Bernstein type argument.
 
 \medskip

 We shall always consider truncated nonlinearities, i.e. $F_k$ such that, for all $A>0$, there exists $k_A>0$ such that
 \be\label{eqabA00}
 F_k(\xi)=|\xi|^p\quad\hbox{ for all $|\xi|\le A$ and $k\ge k_A$}.
 \ee
The following facts will be useful:
for each $t_0\in (0,T^*(\phi))$, there exist $k_0=k_0(t_0)$ and $M=M(t_0)>0$ such that
\be\label{eqabA0}
u_k=u\quad\hbox{ on $[0,t_0]\times \overline\Omega$ for all $k\ge k_0$}
\ee
and
\be\label{eqabA}
|\partial_tu_k|\le M \quad\hbox{ on $[t_0,\infty)\times\Omega$ for all $k\ge k_0$}.
\ee
   Indeed, \rife{eqabA0} is a consequence of \rife{eqabA00} and uniqueness for problem \rife{approxpbm}, 
whereas property \rife{eqabA} easily follows by applying the maximum principle to the equation satisfied by $\partial_tu_k$.

 \vskip1em

Recall the following standard comparison principle.

 \begin{prop}\label{compP}
 Let $\Omega\subset\R^n$ be a bounded domain and $T>0$. Set $Q:=(0,T)\times \Omega$ and 
$$\mathcal{P}v:=v_t-\Delta v-|\nabla v|^p.$$ 
If $v_1,v_2\in C^{1,2}(Q)\cap C(\overline Q)$ satisfy 
$\mathcal{P}v_1\le \mathcal{P}v_2$ in $Q$ and $v_1\le v_2$ on $\partial_PQ$, then $v_1\le v_2$ on $Q$.
 \end{prop}
 
Note that the $v_i$ need not be $C^1$ up to the boundary.
Let us give the short proof for convenience (cf., e.g., \cite{SZ06}).

\begin{proof} Fix $\tau\in (0,T)$, $\eps>0$ and let $w = v_1-v_2-\eps t$.
We see that $w$ cannot attain a local maximum in $(0,\tau]\times \Omega$, since at such a point, we would have 
$0\le  w_t-\Delta w=|\nabla v_1|^p-|\nabla v_2|^p-\eps = -\eps < 0$.
Therefore, $\max_{Q_\tau} w = \max_{\partial_PQ_\tau}w$, and the conclusion follows by letting $\eps\to 0$ and $\tau\to T$.
\end{proof}

We shall also use the following simple comparison principle for the viscosity solution 
of \rife{vhj-1}-\rife{vhj-3}.\footnote{We recall that a comparison principle holds even in the strong form, 
 i.e. for possibly discontinuous {\it generalized} solutions satisfying the boundary conditions in a relaxed sense (see \cite{BDL04}).
 However we will not need this stronger form.}
 
\begin{prop}\label{compP2}
Let $\phi,\psi\in X$ and let $u, v$ be the corresponding
global viscosity solutions of \rife{vhj-1}-\rife{vhj-3}.
If $u(t_0)\le v(t_0)$ in $\overline\Omega$ for some $t_0\ge 0$,
then $u(t)\le v(t)$ in $\overline\Omega$ for all $t\in [t_0,\infty)$.
\end{prop}

   For convenience, we give a short proof, based on the approximation result in Theorem~\ref{prelimprop2}. 

\begin{proof} 
In view of Theorem~\ref{prelimprop2}(ii), this amounts to showing that 
$U\le V$ in $[t_0,\infty)\times\overline\Omega$,
where $U,V$ are the limit solutions given by Theorem~\ref{prelimprop2}(i).
Let $u_k$ be given by Theorem~\ref{prelimprop2}(i). 
 For each $k\ge 0$, we have $u_k(t_0)\le U(t_0)\le V(t_0)$ in~$\overline\Omega$.
Moreover $u_k(t)=0\le V(t)$ on $\partial\Omega$ for all $t\in [t_0,\infty)$.
Since both $u_k$ and $V$ have the required regularity to apply Proposition~\ref{compP},
we deduce that $u_k(t)\le V(t)$ in $\Omega$ for all $t\in [t_0,\infty)$.
Letting $k\to\infty$ it follows that $U(t)\le V(t)$ in $\Omega$ for all $t\in [t_0,\infty)$.
Finally letting $x\to\partial\Omega$, we reach the desired conclusion.
\end{proof}

\section {On the minimal GBU and regularization rate}
\label{sec-rates}

   In this section we first state and prove a slightly more general version of Theorems \ref{minimal_rate} and \ref{minimal_rate2}
concerning the minimal GBU and regularization rates.
Later, at the end of the section, we will give a sufficient condition
for the GBU or regularization rates to be more singular than the minimal ones
(this will be one of the ingredients of the proof of Theorem~\ref{minimal0} in Section~\ref{Sec-min}).

\begin{thm}\label{minimal_rateG}

Let $\phi\in X$, $T>0$ {and let $u$ be the unique viscosity solution of \rife{vhj-1}-\rife{vhj-3}}.

(i) 
Assume that $u$ is a classical solution of \rife{vhj-1}-\rife{vhj-2} 
on the time interval $(T-\delta,T)$ for some $\delta>0$ and
 is not a classical solution on any interval of the form $(T-\delta, T+\eps)$ with $\eps>0$.
Then
$$
\|\nabla u(t)\|_\infty\ge C(T-t)^{-1/(p-2)},\quad t\to T_-.
$$
In particular, this is true with $T=T^*(\phi)$. 

  (ii) 
Assume that $u$ is a classical solution of \rife{vhj-1}-\rife{vhj-2}  
on the time interval $(T,T+\delta)$ for some $\delta>0$ and 
  is not a classical solution on any interval of the form $(T-\eps,T+\delta)$ with $\eps>0$.
Then
$$
\|\nabla u(t)\|_\infty\ge C(t-T)^{-1/(p-2)},\quad t\to T_+.
$$
In particular, this is true with $T=T^r(\phi)$. 
\end{thm}

\medskip

We notice that the assumptions of Theorem~\ref{minimal_rateG}(i) guarantee that 
$$\limsup_{t\to T_-} \|\nabla u(t)\|_\infty=\infty$$
as direct consequence of the local $C^1$ theory.
  The argument does not direcly apply for the regularization time due to the irreversibility of the equation. 
However, it turns out that the analogous property is true under the assumptions of Theorem~\ref{minimal_rateG}(ii),
but this fact is more delicate.  
This is the content  of the following lemma,
which will be useful in the proof of assertion (ii).

\begin{lem}
\label{unbddT}
Under the assumptions of Theorem~\ref{minimal_rateG}(ii), we have
\be\label{lower-bu}
\limsup_{t\to T_+} \|\nabla u(t)\|_\infty=\infty.
\ee
\end{lem}

\begin{proof}
Assume for contradiction that
\be\label{contrad-nablau}
m(t):=\|\nabla u(t)\|_\infty\le M,\quad T<t<T+\tau,
 \ee
for some $M>0$ and $\tau\in {(0,\delta)}$. 
We shall reach a contradiction by estimating {the solutions $u_k$ of the approximate problems \rife{approxpbm} with $F_k(\xi)=\min\bigl(|\xi|^p,k^{p-2}|\xi|^2\bigr)$}.
\smallskip

{\bf Step 1.} We first claim that
\be\label{contrad-nablavk}
m_k(t):=\|\nabla u_k(t)\|_\infty\le C(t-T)^{-1/2},\quad T<t<T+\tau,
\ee
for some $C>0$ independent of $k$.

We prove \eqref{contrad-nablavk} by a Bernstein argument.
Note that $T^*(\phi)<\infty$ by assumption and fix some $t_0\in (0,T^*)$. 
We know from \rife{u_t-bounded} that
$$
\|\partial_tu_k\|_\infty\le A,\quad t\ge t_0,
$$ 
for some constant $A>0$. 
For fixed $k\ge 1$, the function $v=u_k$ saisfies
\be\label{eqvg}
v_t-\Delta v=g(|\nabla v|^2) \quad\hbox{ in $Q:=(T,T+{\delta})\times\Omega$}, 
\ee
where $g(s)=g_k(s)=(\min(s,k))^{(p-2)/2}s$.
Let $z=|\nabla v|^2$. By direct calculation, we see that $z$ {is a strong solution of } 
$$z_t-\Delta z+b\cdot\nabla z +2|D^2v|^2= 0,$$
where $b:=-2g'(z)\nabla v$. 
Using the fact that $|g(z)-v_t|=|\Delta v|\leq \sqrt{n}|D^2v|$, along with the bound $|v_t|\le \tilde C$
(where $\tilde C$ is independent of $k$),
we obtain
$$
n^{-1}z^2\leq n^{-1}g^2(z)\leq 2( |D^2v|^2+ |v_t|^2)\leq 2|D^2v|^2+\tilde C,
$$
hence
$$z_t-\Delta z+b\cdot\nabla z +n^{-1}z^2\le \tilde C.$$
Moreover, combining assumption \eqref{contrad-nablau} with the fact that
$u=0$ on the boundary for $t\in (T,T+\delta)$ and $0\le u_k\le u$, we have
$z=|\partial_\nu v|^2\le |\partial_\nu u|^2\le M^2$ on $(T,T+\tau)\times\partial\Omega$.
The claim now follows by comparing with a supersolution of the form $\overline z:={C_1}(t-T)^{-1}$,
{with $C_1>0$ large enough (independent of $k$). }
\smallskip

{\bf Step 2.} 
Next we claim that
\be\label{contrad-mk}
  \lim_{k\to\infty} m_k(t)=m(t),   \quad T<t<T+{\tau}. 
\ee
Indeed, it follows from \eqref{contrad-nablavk} and parabolic estimates that
$\nabla u_k(t)$ is precompact in $C(\overline\Omega)$ for each $t\in (T,T+{\tau})$. 
Since we already know that 
\be\label{contrad-cvptw}
\hbox{$\nabla u_k(t,\cdot)$ converges to $\nabla u(t,\cdot)$ pointwise in $\Omega$,}
\ee
we deduce that, for each $t\in (T,T+\tau)$,
$\nabla u_k(t,\cdot)\to \nabla u(t,\cdot)$ in $C(\overline\Omega)$,
hence in particular \eqref{contrad-mk}.

\smallskip

{\bf Step 3.}  We shall then show the existence of $\eta>0$ and $k_0$ such that
\be\label{contrad-mk5}
m_k(t)\le M+2 \quad\hbox{ for all $t\in [T-\eta,T]$ and all $k\ge k_0$.}
\ee
To this end, we suitably estimate the oscillation of $m_k$ by means of the variation-of-constants formula.

For any $t_2\in (T,T+{\tau})$, 
by \eqref{contrad-nablau} and \eqref{contrad-mk}, there exists an integer $k_0=k_0(t_2)$ such that 
\be\label{contrad-mk6}
m_k(t_2)\le M+1 \quad\hbox{ for all $k\ge k_0$.}
\ee
For $k\ge k_0$, let 
$$t_1=t_{1,k}:=\min E_k,\quad\hbox{ where $E_k:=\{t\in [t_0,t_2],\, m_k(s)\le M+2 \hbox{ on $[t,t_2]$}\}$}$$
(observe that $E_k$ is nonempty by continuity).
Since $v=u_k$ is a classical solution of \eqref{eqvg},
the function $w:= \partial_tu_k$ is a {strong} solution of 
$$
w_t -\Delta w = {2} g'_k(|\nabla v|^2)\nabla v \cdot \nabla w\,,\qquad t>0,\ x\in \Omega.
$$
For $s,\tau>0$, the variation-of-constants formula for $w$ yields
$$
w(s+\tau)=e^{\tau \Delta}w(s)+
\int_0^\tau e^{(\tau-\sigma)\Delta }\bigl[{2}  g_k'(|\nabla v|^2)\nabla v\cdot\nabla w\bigr](s+\sigma)\, d\sigma.
$$ 
   Here and in the rest of the paper, $(e^{t\Delta})_{t\ge 0}$ denotes the Dirichlet heat semigroup on $\Omega$. 
Define
$$L:=\max_{\sigma\in [0,t_2-t_1]} \sigma^{1/2} \|\nabla w(t_1+\sigma)\|_\infty.$$
Using $m_k(t)\le M+2$ on $[t_1,t_2]$,
{$0\le g_k'(s)\le (p/2)s^{(p-2)/2}$ a.e.~and} 
$\int_0^\tau (\tau-\sigma)^{-1/2}\sigma^{-1/2}\, d\sigma=\int_0^1 (1-z)^{-1/2}z^{-1/2}\, dz$,
it follows that
$$
\begin{aligned}
\|\nabla w(t_1+\tau)\|_\infty
&\leq C\tau^{-1/2}\|w(t_1)\|_\infty+
C\int_0^\tau (\tau-\sigma)^{-1/2}\|\nabla v(t_1+\sigma)\|_\infty^{p-1}\|\nabla w(t_1+\sigma)\|_\infty\, d\sigma\\
&\leq CA\tau^{-1/2}+
C(M+2)^{p-1}\int_0^\tau (\tau-\sigma)^{-1/2}\|\nabla w(t_1+\sigma)\|_\infty\, d\sigma \\
&\leq CA\tau^{-1/2}+C_1L(M+2)^{p-1},
\end{aligned}
$$
where $C,C_1>0$ depend only on $\Omega$. Multiplying by $\tau^{1/2}$ and taking the supremum for $\tau\in [0,t_2-t_1]$, we obtain
\be\label{contrad-mk3}
L\leq CA+C_1(t_2-t_1)^{1/2}L(M+2)^{p-1}.
\ee

Now we claim that
\be\label{contrad-mk4}
t_1\le \bar t_2:=\max\big[t_0,t_2-(2C_1(M+2)^{p-1})^{-2},t_2-(4CA)^{-2}\bigl]. 
\ee
Assume for contradiction that $t_1>\bar t_2$.
In particular, we have $(t_2-t_1)^{1/2} < [2C_1(M+2)^{p-1}]^{-1}$, 
so that \eqref{contrad-mk3} guarantees $L\leq 2CA$. We also have $m_k(t_1)=M+2$ by definition of $E_k$.
Using \eqref{contrad-mk6}, we then get
$$
\begin{aligned}
M+2=m_k(t_1)
&\le m_k(t_2) + \int_0^{t_2-t_1} \| {\nabla v_t}(t_1+\tau)\|_\infty\,d\tau \\
&\le M+1 + 2CA\int_0^{t_2-t_1}\tau^{-1/2}\,d\tau
\le M+1 + 4CA(t_2-t_1)^{1/2}<M+2,
\end{aligned}$$
which is impossible.
Therefore, inequality \eqref{contrad-mk4} is satisfied and \eqref{contrad-mk5} follows by choosing $t_2>T$ close enough to $T$.

\smallskip

{\bf Step 4.} Conclusion.
For each $t\in[T-\eta,T]$, owing to \eqref{contrad-cvptw},
we have $m(t)\le \liminf_{k\to\infty} m_k(t)\le M+2$.
This along with \eqref{contrad-nablau} implies that $u$ is actually a classical solution on $[T-\eta,T+\tau)$, 
contradicting the hypotheses of Theorem~\ref{minimal_rateG}(ii).
The assumption \eqref{contrad-nablau} thus cannot hold and \eqref{lower-bu} is proved.
\end{proof}

{We now proceed to prove Theorem~\ref{minimal_rateG}.
The proof relies on estimate \rife{u_t-bounded} and on the variation-of-constants formula applied to $u_t$,
pushing further the arguments in the proof of Lemma~\ref{unbddT}.
Actually, we here modify the proof from \cite[pp.~369-370]{QS07} of the weaker estimate \rife{lower-bu-weaker}.
Namely we eliminate the possible time oscillations
by considering the supremum of $|\nabla u|$ over time-space cylinders of the form $(s(t),t)\times \Omega$
with $s(t)$ suitably chosen,
and by integrating in time the intermediate estimate \rife{varcost3} below. }

\begin{proof}[Proof of Theorem~\ref{minimal_rateG}.]
Set
$$
I=(T_1,T_2):=\ \ 
\begin{cases}
(T-\delta,T) \quad & \hbox{ in case (i)}\\
\noalign{\vskip 1mm}
(T,T+\delta) \quad & \hbox{ in case (ii).}
\end{cases}
$$
Since $u$ is a classical solution of \rife{vhj-1}-\rife{vhj-2} on the time interval I,
the function $w:= u_t$ is a solution of 
$$
w_t -\Delta w = p|\nabla u|^{p-2}\nabla u \cdot \nabla w\,,\qquad {t\in I,\  x\in \Omega}. 
$$
For $s,t\in I$ {with $s<t$,} 
we may use the variation-of-constants formula for $w$ to write
\be\label{varcost}
w(s+\tau)=e^{\tau\Delta}w(s)+
p\int_0^\tau e^{(\tau-\sigma)\Delta }(|\nabla u|^{p-1}\nabla u\cdot\nabla w)(s+\sigma)\, d\sigma,
\quad \tau\in (0,t-s).
\ee
Also,  fixing any $t_0\in (0,T^*)$, we know from \rife{u_t-bounded} that
\be\label{Beqaa}
\|u_t\|_\infty\le A,\quad t\ge t_0,
\ee
for some constant $A>0$. Without loss of generality, we may assume that $t_0< T_1$. 
\vskip0.4em
Now we define
$$
m(t):=\|\nabla u(t)\|_\infty,\quad t\in I
$$
and
$$
M(s,t):=\max_{\tau\in[s,t]} m(\tau), \qquad 
K(s,t):=\max_{\sigma\in [0,t-s]} \sigma^{1/2} \|\nabla w(s+\sigma)\|_\infty
$$
for $s,t\in I$ {with $s<t$. }
Using \rife{Beqaa} and
$\int_0^\tau (\tau-\sigma)^{-1/2}\sigma^{-1/2}\, d\sigma
=\int_0^1 (1-z)^{-1/2}z^{-1/2}\, dz$,
it follows from \rife{varcost} that
\be\label{varcost2}
\begin{split}
\|\nabla w(s+\tau)\|_\infty
&\leq C\tau^{-1/2}\|w(s)\|_\infty+
C\int_0^\tau (\tau-\sigma)^{-1/2}\|\nabla u(s+\sigma)\|_\infty^{p-1}\|\nabla w(s+\sigma)\|_\infty\, d\sigma\\
&\leq CA\tau^{-1/2}+
CM^{p-1}(s,t)\int_0^\tau (\tau-\sigma)^{-1/2}\|\nabla w(s+\sigma)\|_\infty\, d\sigma\\
&\leq A_1\tau^{-1/2}+C_1K(s,t)M^{p-1}(s,t).
\end{split}
\ee
Multiplying by $\tau^{1/2}$ and taking the supremum for $\tau\in [0,t-s]$, we obtain
\be\label{Kst}
K(s,t)\leq A_1+C_1(t-s)^{1/2}K(s,t)\,M^{p-1}(s,t).
\ee
Let now $t\in I$, and assume in addition that $t\ge T-\delta/2$ in case (i). We notice that, for fixed $t$,  $M(s,t)$ is a continuous function of $s$ which satisfies 
$$
\displaystyle\lim_{s\to (T_1)_+}(t-s)^{1/2}M^{p-1}(s,t)\ \ 
\begin{cases}
\ge (\delta/2)^{1/2}M^{p-1}(T-\delta,T-\delta/2)>0 
&\quad \hbox{ in case (i)}\\
\noalign{\vskip 2mm}
=\infty 
&\quad\hbox{ in case (ii)}\end{cases}
$$
(applying Lemma~\ref{unbddT} in case (ii)),
as well as
$\displaystyle\lim_{s\to t-}(t-s)^{1/2}M^{p-1}(s,t)=0$. So, 
in both cases, we may choose some $s=s(t)\in (T_1,t)$ such that 
\be\label{Mst}
C_1(t-{s(t)})^{1/2}M^{p-1}({s(t)},t)=c_0, 
\ee
with $c_0\in (0,1/2]$ independent of $t$. 
With this choice of $s$,  it follows from \rife{Kst} that
$$ 
K(s(t),t)\leq 2A_1.
$$ 
{In particular, taking $\sigma=t-s$ in the definition of $K(s,t)$,  we get
\be\label{varcost3}
\|\nabla w(t)\|_\infty \le 2A_1(t-s(t))^{-\frac12}.
\ee
}

On the other hand, a standard argument (see e.g. \cite[p.~454]{QS07}) shows that
$m(t)$ is locally Lipschitz and satisfies 
\be\label{Beqc}
|m'(t)|\le \|\nabla u_t(t)\|_\infty,\quad a.e.\  t\in I\,.
\ee
Hence
$$
|m'(t)|\le C (t-s(t))^{-1/2},\quad a.e.\  t\in \tilde I\,, 
$$
where $\tilde I = (T-\frac\de2, T)$ in case (i) and $\tilde I=I$ in case (ii).
By integration, for $\tau\in [s(t),t)$, we get
$$m(\tau)=m(t)-\int_\tau^t m'(\sigma)\,d\sigma
\le m(t)+C\int_{s(t)}^t (\sigma-s(t))^{-1/2}\,d\sigma
\le m(t)+C (t-s(t))^{1/2}, 
$$
hence 
$$
M(s(t),t)\le m(t)+C.
$$
Now going back to \rife{varcost3}, and using \rife{Mst}, we conclude that
$$
\|\nabla w(t)\|_\infty\le {2A_1} (t-s(t))^{-1/2} = {2A_1} \, c_0^{-1}C_1 \,M^{p-1}(s(t),t)\le C(m(t)+1)^{p-1}. 
$$
Property \rife{Beqc} then guarantees that
$$
|m'(t)|\le C(m(t)+1)^{p-1},\quad a.e.\  t\in \tilde I.
$$
The desired estimates then follow by integration, using the fact that 
$\limsup_{t\to T_\pm}m(t)=\infty$.
\end{proof}

\medskip

By modifying the proof of Theorem~\ref{minimal_rateG},
we obtain the following proposition, which gives a sufficient condition
for the GBU rate to be more singular than the minimal one.
We also give a similar property in case of immediate regularization after $T^*$.
This will be one of the ingredients of the proof of Theorem~\ref{minimal0} in Section~\ref{Sec-min}. 
{We here denote $\delta(x):= {\rm dist}(x,\partial\Omega)$.}

\begin{prop}\label{contut}
Let $\phi\in X$ and assume $T^*<\infty$.
\vskip0.4em
(i) Assume that $u_t$ is continuous at $\{T^*_-\}\times\partial\Omega$, i.e.:
\be\label{continuityutA}
\lim_{t\to T^*_-,\,\delta(x)\to 0}u_t(t,x)=0.
\ee
Then
$$(T^*-t)^{1/(p-2)} \|\nabla u(t)\|_\infty\to\infty,\quad t\to T^*_-.$$
\vskip0.4em
(ii) Assume that $u$ becomes a classical solution 
again (including boundary conditions)
on some interval $(T^*,T^*+\delta)$. Assume in addition that $u_t$ is continuous at $\{T^*_+\}\times\partial\Omega$, i.e.:
\be\label{continuityutB}
\lim_{t\to T^*_+,\,\delta(x)\to 0}u_t(t,x)=0.
\ee
Then
$$
(t-T^*)^{1/(p-2)} \|\nabla u(t)\|_\infty\to\infty,\quad t\to T^*_+.
$$
\end{prop}

\begin{proof} 
We slightly modify the proof of Theorem~\ref{minimal_rateG}. We use the same notations where, we recall, $w(t,x):= u_t(t,x)$.
Now, in  \rife{varcost},  we estimate $e^{t\Delta }w(s)$ by taking advantage of assumptions
\rife{continuityutA} and \rife{continuityutB}.
Fix $\eps>0$. There exists $a=a(\eps)>0$ and an interval 
$I_\eps=[T_\eps,T^*)$ in case (i) and $I_\eps=[T^*,T_\eps)$ in case (ii), such that
$$\sup_{\Omega_a} |w(s,x)|\le \eps,\quad s\in I_\eps,\quad\hbox{ where $\Omega_a=\{x\in \Omega;\, \delta(x)\le a\}$}.$$
Denote by $G(t,x,y)$ the Dirichlet heat kernel of $\Omega$ and recall the Gaussian estimate:
$$|\nabla_x G(t,x,y)|\le c_1t^{-(n+1)/2}\exp\Bigl[-c_2{|x-y|^2\over t}\Bigr].$$
Using \rife{continuityutA}, \rife{continuityutB}, \rife{Beqaa}, we may then write, 
for all $s\in I_\eps$, {$\tau\in (0,1)$} 
 and $x\in \Omega_{a/2}$,
\begin{align*}
\bigl|\nabla  [e^{\tau \Delta }w(s)](x)\bigr|
&=\Bigl|\int_\Omega \nabla_x G(\tau,x,y)w(s,y)\, dy\Bigr| \\
&\le c_1\eps\int_{\Omega_a}  \tau^{-\frac{n+1}{2}}\exp\Bigl[-c_2{|x-y|^2\over \tau}\Bigr]\, dy
+c_1A\int_{\Omega\setminus \Omega_a}  \tau^{-\frac{n+1}{2}}\exp\Bigl[-c_2{|x-y|^2\over \tau}\Bigr]\, dy\\
&\le c_3\eps \tau^{-1/2}
+c_1A \tau^{-\frac{n+1}{2}}\exp\Bigl[{-c_2a^2\over 4\tau}\Bigr].
\end{align*}
On the other hand, for all $s\in I_\eps$, { $\tau\in (0,1)$ }
 and $x\in \Omega\setminus\Omega_{a/2}$, we have $|\nabla w(s+\tau,x)|\le C(\eps)$,
since the solution remains smooth for all times away from the boundary.
Putting this together, we obtain, 
for all $s\in I_\eps$, {$t\in (s,s+1)$ and $\tau\in (0,t-s)$, } 
$$\|\nabla w(s+\tau)\|_\infty
\le c_3\eps \tau^{-1/2}
+c_1A \tau^{-\frac{n+1}{2}}\exp\Bigl[{-c_2a^2\over 4\tau}\Bigr]+C(\eps)+C_1K(s,t)M^{p-1}(s,t).$$
Multiplying by $\tau^{1/2}$, taking the supremum for $\tau\in [0,t-s]$, 
and assuming $t-s\le \tau_0(\eps)$ with $\tau_0(\eps)>0$ sufficiently small, we obtain
$$K(s,t)\leq 2c_3\eps+C_1(t-s)^{1/2}K(s,t)M^{p-1}(s,t).$$

Let now $t\in I$. Choosing $s=s(t)$ as in \rife{Mst}, we see that the conditions $s\in I_\eps$ and $t-s\le \tau_0(\eps)$
are satisfied whenever $t>T^*-\delta_\eps$ (resp., $t<T^*+\delta_\eps$), with $\delta_\eps$ sufficiently small.
Indeed, in case~(i) this follows from
$M(s,t)\ge m(t)$ and $\lim_{t\to T^*_-}m(t)=\infty$, 
whereas in case (ii) this is just a consequence of $T^*<s<t$.
Consequently, we obtain
$$K(s(t),t)\leq 4c_3\eps.$$ 

Arguing as before, we end up with 
$$
|m'(t)|\le c_4\eps(m(t)+1)^{p-1},
\quad\hbox{ for a.e. $t\in(T^*-\delta_\eps,T^*)$ \ (resp., $t\in (T^*,T^*+\delta_\eps))$.}
$$
Integrating and using the fact that 
$\limsup_{t\to T^*_\pm}m(t)=\infty$, we obtain
$$
m(t)+1\ge [(p-2)c_4\eps(T^*-t)]^{-1/(p-2)} 
\quad\hbox{ for a.e. $t\in(T^*-\delta_\eps,T^*)$}
$$
(and similarly in case (ii)).
Since $\eps>0$ was arbitrarily small, the conclusion follows. 
\end{proof}

\section{  Preliminary results in one space dimension} 
\label{prelim1d}

  In this section, we state and prove a number of useful preliminary properties of the solution in one space-dimension,
  with $\Omega=(0,1)$.  {Let us recall the notation}
$$
X_1:=\{\phi\in C^1([0,1]),\ \phi(0)=\phi(1)=0\}.
$$
Hence, for $\phi\in X_1$, we consider   the unique global viscosity solution $u$ of the problem
\be\label{vhj1}
\begin{cases}
u_t-u_{xx} =|u_x|^p, & \quad t>0,\ 0<x<1,\\
u =0, & \quad t>0,\ x\in\{0,1\},\\
u(0,x) =\phi(x),  & \quad 0<x<1.
\end{cases}
\ee

We will here consider the specific approximation
\be\label{app-1d}
\begin{cases}
u_{k,t}-u_{k,xx} =F_k(u_{k,x}), & \quad t>0,\ 0<x<1,\\
u_k =0, & \quad t>0,\ x\in\{0,1\},\\
u_k(0,x) =\phi(x), & \quad 0<x<1,
\end{cases}
\ee
where the nonlinearity    $F_k\in C^2(\R)$    is given by the truncation of $|s|^p$ by its second order Taylor expansion at $s=\pm k$.
More precisely, $F_k$ is the even function defined by 
\be\label{eqaaa}
F_k(s)=\begin{cases}
s^p,\quad 0\le s\le k, \\
\noalign{\vskip 1mm}
    k^p+pk^{p-1}(s-k)+p(p-1)k^{p-2}\displaystyle{(s-k)^2\over 2},\quad s>k.  
\end{cases}
\ee
Note that $0\le F_k(s)\le |s|^p$ 
and $F_k$ is convex.    Also, it is easy to check that $F_k$ satisfies
the assumptions of Theorem~\ref{prelimprop2} with $\theta=1/2$, as well as  
\be\label{eqaa}
2F_k(s)\le sF_k'(s)\le pF_k(s),\quad s\ge 0.  
\ee

We start with some basic facts.

\begin{lem}\label{basic-prop0} 
Let $\phi\in X_1$.
\vskip0.3em
(i) Let $t_0\in (0,T^*(\phi))$ and let the constant $M=M(t_0)$ be given by \rife{u_t-bounded}.
Then for all $t\ge t_0$, we have $u_{xx}\le M$ and the function $x\to u(t,x)-\frac{M}{2}x^2$ is concave in $(0,1)$.
  Moreover, the limits 
\be\label{controlnormux0}
\displaystyle\lim_{x\to 0}u_x(t,x)\in \R\cup\{\infty\}\quad\hbox{and}\quad
\displaystyle\lim_{x\to 1}u_x(t,1)\in \R\cup\{-\infty\}
\ee
exist. Whenever they are finite, 
$u_x(t,0)$ and $u_x(t,1)$ exist (and respectively coincide with the limits in \eqref{controlnormux0}) and we have
\be\label{controlnormux}
\|u_x(t)\|_\infty\le \max(u_x(t,0),-u_x(t,1))+M.
\ee
\vskip0.3em
(ii)   Assume in addition that $\phi\not\equiv 0$ is symmetric and nondecreasing on $(0,1/2)$. Then, 
  for all $t>0$, the functions $u_k(t,\cdot)$ ($k\ge 1$) and $u(t,\cdot)$ are symmetric and nondecreasing on $(0,1/2)$. Moreover, we have 
\be\label{utcenter}
u_t(t,1/2)<0,\quad t>0,
\ee
and
\be\label{uktcenter}
u_{k,t}(t,1/2)<0,\quad t>0.
\ee
\end{lem}

\begin{proof} 
(i) By \rife{u_t-bounded}, we have $u_{xx}=u_t-|u_x|^p\le M$ in $[t_0,\infty)\times (0,1)$
and the first part of the assertion immediately follows.
  Since $u_x(t,x)-Mx$ is nonincreasing, the limits in \eqref{controlnormux0} exist.
Since $u(t,\cdot)\in C([0,1])\cap C^1(0,1)$, 
the assertion after \eqref{controlnormux0} follows 
and we have
$$u_x(t,1)-M\le u_x(t,x)-Mx\le u_x(t,0),\quad 0\le x\le 1,$$
which yields \eqref{controlnormux}.
\smallskip

(ii)    The symmetry of the $u_k$ is guaranteed by their uniqueness.
Their monotonicity property is an immediate consequence of the maximum principle applied to the equation for $u_{k,x}$.
Both properties are inherited by $u$ after passing to the limit $k\to\infty$.

Let us show \eqref{utcenter}. Since $u$ is smooth in $(0,\infty)\times(0,1)$, we deduce from the strong  maximum principle that $u_x>0$ for $x\in (0,1/2)$.
Since also $u_x(t,1/2)=0$,
and since the equation for $u_x$ in $(0,\infty)\times(1/4,1/2)$ 
has smooth bounded coefficients,
we may apply the Hopf lemma  to get $u_{xx}(t,1/2)<0$, hence the conclusion. 
The proof of \eqref{uktcenter} is similar. 
\end{proof}

  Our next lemma gives useful bounds on $u_x$.
We note that the bounds in Lemmas \ref{basic-bounds}--\ref{basic-prop} have been already known for classical solutions (see, e.g., \cite{CG96}, \cite{ARS04}),
but since we here deal with viscosity solutions, with possible loss of boundary conditions, it is safer to prove exactly what we need.

\begin{lem}\label{basic-bounds} 
Let $\phi\in X_1$ with $T^*(\phi)<\infty$. Then, for all $t\ge t_0>0$, we have
\be\label{SGBUprofileUpperEst}
u_x(t,x) \le \bigl[\bigl(u_x(t,y)-My\bigr)_+^{1-p}+(p-1)(x-y)\bigr]^{-1/(p-1)}+\, Mx, \quad 0<y<x<1,
\ee
\be\label{SGBUprofileUpperEst2}
u_x(t,x) \le U_*'(x)+\, Mx, \quad 0<x<1,
\ee
\vskip -7pt
\be\label{SGBUprofileUpperEst3}
u_x(t,x) \ge -U_*'(1-x)-\, M(1-x), \quad 0<x<1,
\ee
and
\be\label{SGBUprofileLowerEst}
(u_x(t,x))_+ \ge \bigl[\bigl((u_x(t,y))_++My\bigr)^{1-p}+(p-1)(x-y)\bigr]^{-1/(p-1)}-\, Mx, \quad 0<y<x<1, 
\ee
where $M=M(t_0)$ is given by \rife{u_t-bounded}
and where the reference singular profile   $U_*$ is defined in \eqref{defUstar}.
\end{lem}

\begin{proof} 
For fixed $t\ge t_0$, let $z(x)=(u_x(t,x)-M x)_+$. The function $z$ satisfies
$$z'+z^p=(u_{xx}-M)\chi_{\{u_x>M x\}}+(u_x-Mx)_+^p,
\quad\hbox{for a.e. $x\in (0,1)$.}$$
For each $x$ such that $u_x(t,x)>Mx$, 
we have $(z'+z^p)(x)\leq (u_{xx}-M+|u_x|^p)(x)\le 0$ by \rife{u_t-bounded}. 
Therefore, we have 
\be\label{SGBUprofileUpperEstAux}
z'+z^p\le 0\quad\hbox{ a.e. on $(0,1)$.}
\ee
 In particular $z$ is nonincreasing on $(0,1)$. 
 We may assume that $E:=\{x\in (0,1);\, z(x)>0\}\neq\emptyset$
(otherwise there is nothing to prove) and, letting $a=\sup E\in (0,1]$, we have $z>0$ on $(0,a)$. 
For each $x\in(0,a)$, by integrating
\eqref{SGBUprofileUpperEstAux} we get $z^{1-p}(x)\ge z^{1-p}(y)+(p-1)(x-y)$
for all $y\in (0,x)$.
Since $z\le 0$ on $[a,1)$, we obtain inequality \eqref{SGBUprofileUpperEst}.

\smallskip
Next, inequality \eqref{SGBUprofileUpperEst2} follows by letting $y\to 0$ in \eqref{SGBUprofileUpperEst},
and \eqref{SGBUprofileUpperEst3} follows by applying \eqref{SGBUprofileUpperEst2} to the solution $v(t,x):=u(t,1-x)$.

\smallskip
To prove \eqref{SGBUprofileLowerEst} we now let 
$z(x)=(u_x(t,x))_++\,Mx.$
The function $z$ satisfies
$$
z'+z^p=u_{xx} \chi_{\{u_x>0\}}+M+\bigl[(u_x(t,x))_++\,Mx\bigr]^p 
\ge (u_{xx}+|u_x|^p) \chi_{\{u_x>0\}}+M\ge 0$$
a.e.~on $(0,1)$ by \rife{u_t-bounded}. 
By integration, noting that $z>0$ on $(0,1)$, we get 
$z^{1-p}(x)\le z^{1-p}(y)+(p-1)(x-y)$
and inequality \eqref{SGBUprofileLowerEst} follows. 
\end{proof}

The next result guarantees that unboundedness of $u_x$ near a given time $t\ge T^*$
implies that the space behavior at that time is described by the reference singular profile   $U_*$.

\begin{lem}\label{basic-prop} 
  Let $\phi\in X_1$ with $T^*(\phi)<\infty$.  
  Fix any $t_0\in (0,T^*)$ and let $M(t_0)$ be given by \rife{u_t-bounded}. 
If, for some   $t\ge T^*$,  there exists a sequence $(t_j, x_j)\to (t,0)$ such that $u_x(t_j,x_j)\to \infty$, then 
\be\label{profile2}
|u(t,x)-u(t,0)-U_*(x)| \le K \, {\frac{x^2}2}, \qquad 0< x\le 1/2,
\ee
\be\label{profile}
|u_x(t,x)-U_*'(x)| \le K x, \qquad 0<x\le 1/2,\ee
  and
\be\label{profile3}
|u_{xx}(t,x)-U_*''(x)| \le K, \qquad 0<x\le 1/2,\ee
for some constant $K>0$   depending only on $p$ and $M(t_0)$. 
\end{lem}

\begin{proof} 
Let us first show \rife{profile}.
Since the upper bound in \rife{profile} is guaranteed by Lemma~\ref{basic-bounds}, we only need to show the lower bound.
To this end, writing \eqref{SGBUprofileLowerEst} with $t=t_j$, $y=x_j$ and letting $j\to\infty$, we have
\be\label{SGBUprofileUpperEstAux2}
(u_x(t,x))_+ \ge ((p-1)x)^{-1/(p-1)}-\, Mx, \quad 0<x\le 1/2,
\ee
with $M=M(t_0)$. Since the RHS of \eqref{SGBUprofileUpperEstAux2} is positive for $x<x_0:=(M^{1-p}/(p-1))^{1/(p-2)}$, we get
\be\label{SGBUprofileUpperEstAux3}
u_x(t,x) \ge U_*'(x)-\, Mx, \quad 0<x<x_0.
\ee
On the other hand, by \eqref{SGBUprofileUpperEst3}, we have
$$u_x(t,x)\ge -U_*'(1/2)-\, M, \quad x_0\le x\le 1/2.$$
Now choosing
$$K=K(p,M):=\max\bigl[M,x_0^{-1}(U_*'(1/2)+U_*'(x_0)+M)\bigr],$$
we get
$$
-U_*'(1/2)-\, M\ge U_*'(x_0)-\, Kx_0\ge U_*'(x)-\, Kx, \quad x_0\le x\le 1/2.
$$
This together with \eqref{SGBUprofileUpperEstAux3} yields the lower estimate on $u_x$ in \rife{profile}.

\smallskip
To show \rife{profile3}, we use \rife{profile} to  write:
\be\label{SGBUprofileUpperEstAux4}
|u_{xx}(t,x)-U_*''(x)|\le |u_t|+\bigl||U_*'|^p-|u_x|^p\bigr|\le |u_t|+|U_*'+Kx|^p-U_*'^p, \quad 0<x\le 1/2.
\ee
On the other hand, we have
\be\label{SGBUprofileUpperEstAux5}
|U_*'+Kx|^p-U_*'^p=U_*'^p\bigl[(1+cKx^{p/(p-1)})^p-1\bigr]\le cKU_*'^px^{p/(p-1)}=cK,
\ee
where $c=c(p)$ denotes a generic positive constant.
Property \rife{profile3} then follows from \eqref{SGBUprofileUpperEstAux4}, 
\eqref{SGBUprofileUpperEstAux5} and \rife{u_t-bounded}.

\smallskip
Finally, \rife{profile2} follows from \rife{profile} by integration.
\end{proof}

We now point out  that, should the solution $u$ have lost the boundary condition at some $t_0>0$,
 then necessarily the gradient must blow up near the boundary at this time. This is a consequence of general results of \cite{BDL04}.
However, we shall provide a direct and more elementary proof in the next lemma.
\medskip

\begin{lem}\label{bdl} 
  Let $\phi\in X_1$ with $T^*(\phi)<\infty$  
and assume that, for some $t>T^*$, we have
\be\label{HypInfiniteDeriv}
 u(t,0)>0.
 \ee 
Then we have 
\be\label{InfiniteDeriv}
\lim\limits_{ x \to 0_+} u_x(t,x) =\infty.
\ee
\end{lem}

\begin{proof} 
  Fix some $t_0\in (0,T^*(\phi))$ and let $M=M(t_0)$ be given by \rife{u_t-bounded}. 
Set $V(x)=u(t,x)$, $V_k(x)=u_k(t,x)$. 
  By Lemma~\ref{basic-prop0},  the function $V_x-Mx$ is nonincreasing  and has a limit (finite or $+\infty$) as $x\to 0$.
\medskip
Assume for contradiction that this limit is finite.  Then there exists $A>0$ such that 
 $$
 V_x(x)\le A+Mx\le A+M,\quad 0<x<1.
 $$
 Let $z(x)=(V_{k,x}(t,x))_++\,Mx.$
Then, owing to \rife{eqabA}, for all $k\ge k_0$, the function $z$ satisfies
$$
z'+z^p=V_{k,xx} \chi_{\{V_{k,x}>0\}}
 +M+\bigl[(V_{k,x})_++\,Mx\bigr]^p  
\ge (V_{k,xx}+|V_{k,x}|^p) \chi_{\{u_{k,x}>0\}}+M\ge 0
$$
almost everywhere on $(0,1)$. 
By integration, noting that $z>0$ on $[0,1]$, we get 
$$
z^{1-p}(y)\le z^{1-p}(x)+(p-1)(y-x),
\quad 0<x<y\le 1/2.
$$ 
Now fix $y_0\in (0,1/2)$ such that $y_0\le {(A+2M+1)^{1-p}\over 2(p-1)}$.
For $k\ge k_1$ sufficiently large, we have
$$
V_{k,x}(y_0)\le V_x(y_0)+1\le A+M+1,
$$
hence
$$
\bigl((V_{k,x})_+(x)+\,Mx\bigr)^{1-p} 
\ge (A+2M+1)^{1-p}-(p-1)y_0\ge \textstyle\frac12(A+2M+1)^{1-p},
\quad 0<x<y_0.
$$
Therefore,
$$
V_{k,x}(x)\le C:={2^{\frac1{p-1}}}(A+2M+1),
\quad 0<x<y_0.
$$
Since $V_{k}(0)=0$, by integration we get
$$
V_{k}(x)\le Cx,
\quad 0<x<y_0.
$$
Passing to the limit $k\to\infty$   for each fixed $x\in(0,y_0)$,  we get
$V(x)\le Cx$ for $0<x<y_0$.   Then passing to the limit $x\to 0$, we finally obtain
$u(t,0)=V(0)=0$, contradicting the loss of boundary condition assumption \eqref{HypInfiniteDeriv}.
\end{proof}

As a key tool in order to 
  study the regularization after $T^r$ (respectively $T^*$) for nonminimal (respectively  minimal) solutions, 
we shall perform an analysis in terms of the distance of $u(t,\cdot)$ to the reference profile $U_*$. 
The next result is a regularizing barrier argument showing how the smoothing rate can be related to the gap between $U_*$ and $u$.

\begin{lem}\label{barrier}
Let $\phi\in X_1$ with $\phi$   symmetric.  {We set $\alpha= \frac1{p-1}$.}
Let $T>0$, $m\in [2, 3-\alpha)$ and assume that
\be\label{BeqabAAa}
u(T,x)\le U_*(x)-bx^m, \quad 0<x<1,
\ee
for some $b>0$. Then $u$ is a classical solution of problem \rife{vhj1} (including boundary conditions)
for all $t\in(T,\infty)$. Moreover, we have
\be\label{BeqabAAb}
\|u_x(t)\|_\infty\le C(t-T)^{-\gamma},\quad t\to T_+,
\quad\hbox{ where }
\gamma={\alpha\over 3-m-\alpha}.
\ee
We note that $\gamma={1\over p-2}$ for $m=2$ and that
$\gamma={1\over (3-m)(p-1)-1}>{1\over p-2}$ for $2<m<3-\alpha$.
\end{lem}

\begin{rem}  \label{RemSepar}
  Lemma~\ref{barrier}  shows that the quadratic separation is sufficient in order to have the minimal smoothing rate (${1\over p-2}$ as in Theorem~\ref{minimal_rate}) and that no improvement on \rife{BeqabAAb} can be expected 
if we make the stronger assumption \rife{BeqabAAa} with some $m<2$ (for instance $m=1$).
This will also be made apparent in the following proof.
\end{rem}

\begin{proof}[Proof of Lemma~\ref{barrier}] 
In this proof we shall a priori consider the whole range of values $m\in [1,\infty)$,
in order to understand the phenomenon described in Remark~\ref{RemSepar}. 

Shifting the origin of time at $t=T$   and denoting $U=U_*$,  we look for a supersolution of the form 
$$z(t,x):=U(x+a(t))-U(a(t))-bx^m,\quad 0\le t\le\delta,\ 0\le x\le 1,$$
for some $\delta>0$ and some function $a\in C([0,\delta])\cap C^1((0,\delta])$
with $a(0)=0$ and $a(t)>0$.
Writing $a$ for $a(t)$, we compute
\begin{align*}
Pz
&:=z_t-z_{xx}-|z_x|^p \\
&=[U'(x+a)-U'(a)]a'(t)-U''(x+a)+bm(m-1)x^{m-2}-|U'(x+a)-mbx^{m-1}|^p \\
&=[U'(x+a)-U'(a)]a'(t)+bm(m-1)x^{m-2}+[{U'}^p(x+a)-|U'(x+a)-mbx^{m-1}|^p].
\end{align*}
Assuming $a\le 1$ and   (without loss of generality)  $b<b_0(p)$ with $b_0(p)>0$ small enough,
we have $mbx^{m-1}\alpha^{-\alpha}(x+a)^{\alpha}\le 1/2$ for all $x\in (0,1)$.
We then compute
\begin{align*}
{U'}^p(x+a)-|U'(x+a)-mbx^{m-1}|^p
&=\alpha^{p\alpha}(x+a)^{-p\alpha}-\bigl[\alpha^{\alpha}(x+a)^{-\alpha}-mbx^{m-1}\bigr]^p\\
&=\alpha^{p\alpha}(x+a)^{-p\alpha} \bigl[1- \bigl[1-mbx^{m-1}\alpha^{-\alpha}(x+a)^{\alpha}\bigr]^p\bigr] \\
&\ge\alpha^{p\alpha}(x+a)^{-p\alpha} 
\bigl[2^{1-p}pmbx^{m-1}\alpha^{-\alpha}(x+a)^{\alpha}\bigr]\\
&=2^{1-p}pmbx^{m-1}\alpha(x+a)^{-1},
\end{align*}
where we used $(p-1)\alpha=1$ and 
$(1-t)^p\leq 1-2^{1-p}pt$ for  $t\in (0,1/2)$.

The inequality $Pz\ge 0$ in $(0,\delta]\times (0,1)$ will thus be ensured provided we have
\be\label{BeqabAA}
[U'(a)-U'(x+a)]a'(t)\le bm(m-1)x^{m-2}+2^{1-p}pmbx^{m-1}\alpha(x+a)^{-1}.
\ee
We distinguish now two cases according to whether $m=1$ or $m>1$.

\begin{itemize}

\item For $m=1$, \rife{BeqabAA} reduces to
$$
a'(t)\le {2^{1-p}pb{\alpha}\over \overline D_a},
\quad\hbox{ where 
$\overline D_a=\displaystyle\sup_{x\in (0,1)}D_a(x)$, 
\quad
$D_a(x):=[U'(a)-U'(x+a)](x+a)>0$}.
$$
{Since $U'$ is a decreasing function, } 
we see that $D_a(x)$ is an increasing function of $x$, 
hence $\overline D_a=D_a(1)=[U'(a)-U'(1+a)](1+a)\le 2U'(a)$
and a sufficient  condition for $Pz\ge 0$ is given by
$$a'(t)\le {2^{-p}pb{\alpha}\over U'(a)}=c(p)b a^\alpha.$$
We thus choose 
\be\label{BeqabA}
a(t)=\eta\, t^{1/(1-\alpha)} 
\ee
with $\eta$ sufficiently small.

\item For $m>1$, we use only the first term on the RHS of 
\rife{BeqabAA}, since it dominates the second one.
Condition \rife{BeqabAA} then reduces to

\be\label{lookfora}
a'(t)\le {m(m-1)b\over \overline D_a},
\quad\hbox{ where $D_a(x):=[U'(a)-U'(x+a)]x^{2-m}>0$}.
\ee
Observe that we have $\overline D_a\ge D_a(1)=U'(a)-U'(1+a)\ge c(p)U'(a)$.
\medskip

$\bullet$ If $1<m\le 2$, then $\overline D_a\le U'(a)$
and we would make the same choice as in \rife{BeqabA}.

\medskip
$\bullet$ Next consider the case $m>2$.
If $0<x\le a$, then $D_a(x)=U''(a+\theta x)x^{3-m}$ for some $\theta(x)\in (0,1)$.
{In particular, notice that if $m>3$ we have $\overline D_a=\infty$ and the inequality \rife{lookfora} is impossible. Thus we already need to restrict to  $2<m\le 3$. Now, }for $0<x\le a$, we have $D_a(x)\le c\,U''(a)a^{3-m}=c\,a^{2-\alpha-m}$.
For $a<x<1$, we have $D_a(x)\le U'(a)a^{2-m}=c\,a^{2-\alpha-m}$.
Therefore we have $c_0a^{2-\alpha-m}\le \overline D_a\le c_1\, a^{2-\alpha-m}$ for different constants $c_0,c_1$. 
So we need a  function $a$ such that $a'(t)\le c\,a^{\alpha+m-2}$.
However, since $a(0)=0$ and $a(t)>0$, this induces the further restriction
$\alpha+m-2<1$ i.e., $m<3-\alpha$ and  leads to the choice
\be\label{BeqabB}
a(t)=c\,t^{1/(3-m-\alpha)}.
\ee 
{Incidentally, we notice how the above computations offer a motivation for  the restriction $m\in [2,3-\alpha)$ required in our assumptions.} 
\end{itemize}

{Now,}  for each $k\ge 1$, since $z(0,x)=U(x)-bx^m\ge u(T,x)\ge u_k(T,x)$
and $z\ge 0=u_k$ at $x\in\{0,1\}$   (decreasing $b>0$ if necessary),   the comparison principle yields $z(t,\cdot)\ge  u_k(T+t,\cdot)$,
hence $z(t,\cdot)\ge  u(T+t,\cdot)$, 
{for all $t\in (0,\delta]$.}
In particular, $u$ satisfies the boundary conditions in the classical sense for all {$t\in (T,T+\delta]$.} 
Moreover, we have
\be\label{BeqabA2}
u_x(T+t,0)\le z_x(t,0)=U'(a(t))=c\, a^{-\alpha}(t),\quad 0<t<\delta.
\ee
  Fix $0<t_0<\min(T^*,T)$ and let $M=M(t_0)$ be given by \rife{u_t-bounded}.  Combining \rife{BeqabA2} with \rife{BeqabA} or \rife{BeqabB},
{estimate  \rife{BeqabAAb}} 
 follows from symmetry and the fact that $u_x-Mx$ is nonincreasing in $x$ by Lemma~\ref{basic-prop0}.
 
{Finally, since $z_t\le 0$ hence $u(T+t,x)\le z(t,x)\le U(x)-bx^m$ on $(0,\tau]\times [0,1]$, we may repeat the comparison on each time interval
$[T+j\delta,T+(j+1)\delta]$, with $j$ integer, to deduce that $u(t,x)\le U(x)-bx^m$ on $[T,\infty)\times [0,1]$.
The argument in the preceding paragraph then guarantees that $u$ is a classical solution (including boundary conditions) on $(T,\infty)\times [0,1]$.}
\end{proof}

In the next elementary ODE lemma, we transform   an information on the behavior of $u_t$ near the boundary
(such information will be obtained in Section~\ref{Sec-nonmin}) 
into a separation property of $U_*-u$.

\begin{lem}\label{separation} Let $\phi\in X_1$, $T>0$ and $\ell\ge 0$.
\vskip0.4em
(i) Assume that
\be\label{Beqdd0}
u_t(T,x)\ge -bx^\ell,\quad 0<x\le   1/2 
\ee
for some $b>0$, along with
\be\label{Beqdd}
\lim_{x\to 0}u_x(T, x)=\infty.
\ee
Then there exist constants $c_1, c_2>0$ such that
$$
u(T,x)\ge u(T,0)+U_*(x)-c_1x^{\ell+2}
\quad\hbox{ and }\quad
u_x(T,x)\ge U'_*(x)-c_2x^{\ell+1},
\qquad 0<x\le 1/2.
$$
\vskip0.4em
(ii) Assume that 
\be\label{HypLemmeODE}
u_t(T,x)\le -bx^\ell,\quad 0<x\le   1/2  
\ee
for some $b\geq 0$. Then we have
$$
u(T,x)\le u(T,0)+U_*(x)-c_1\, x^{\ell+2}
\quad\hbox{ and }\quad
u_x(T,x)\le U'_*(x)-c_2\, x^{\ell+1},
\qquad 0<x\le 1/2,
$$
  for some constants $c_1, c_2>0$ if $b>0$, or with $c_1=c_2=0$ if $b=0$. 
\end{lem}

\vskip0.4em
\begin{rem} We will actually use assertion (ii) in the case $b=0$ only. 
 In particular this says that  whenever $u_t(t,x)\leq 0$ and $u$ satisfies the boundary condition $u(t,0)=0$, 
then $u$ lies below the singular profile $U_*(x)$.
However we present the general conclusion of part (ii) (with, possibly, $b>0$), as it shows the optimality of the estimate in part (i).
\end{rem}

\begin{proof}
Set $V(x)=u_x(T,x)$.
\smallskip

(i)   By  \eqref{Beqdd} and Lemma \ref{basic-prop} (see \rife{profile}), we have $V(x)\ge U_*'(x)-c\,x$ for all $x\in (0,1/2]$. {Here and below, $c$ will denote a generic positive constant. } 
In particular there exists $x_0\in (0,1/2]$ such that 
\be\label{Beqe}
V(x)\ge \textstyle\frac12 U_*'(x)>0,\quad 0<x\le x_0.
\ee
By assumption \eqref{Beqdd0}, we thus have 
\be\label{BeqdA}
-V'=V^p-u_t(T,x)\le V^p+bx^\ell,\quad 0<x\le x_0.
\ee
Next using \rife{BeqdA}, \rife{Beqe} and recalling $p\alpha=\alpha+1$, {we have}
$$
{1\over p-1}(V^{1-p})' =-V'V^{-p}  \le 1+bx^\ell V^{-p}
\le 1+c\,x^{\ell+\alpha+1},\quad 0<x<{x_0}.  
$$
Integrating and using assumption \rife{Beqdd} again, we get
$$
V^{1-p}(x)\le (p-1)x[1+c\,x^{\ell+\alpha+1}],\quad 0<x<{x_0},
$$
hence
$$
V(x) \ge U_*'(x)[1+c\,x^{\ell+\alpha+1}]^{-\alpha}
\ge U_*'(x)[1-c\,x^{\ell+\alpha+1}]=U_*'(x)-c\,x^{\ell+1},\quad 0<x<{x_0}.
$$
In view of {\rife{profile}}, 
  decreasing $c$ if necessary,   we deduce 
$$
V(x) \ge U_*'(x)-c\,x^{\ell+1},\quad 0<x\le 1/2,
$$
and assertion (i) follows by a further integration.

\vskip1em

(ii)   Since $-V'=|V|^p-u_t(T,x)\ge 0$ owing to \rife{HypLemmeODE}, {we have that $V$ is decreasing in $(0,1/2)$; hence, due to \rife{Beqdd}, 
there exists ${x_1}\in (0,1/2]$ such that
$V>0$ in $(0,{x_1})$ and $V\le 0$ in $(x_1,1/2]$.}

By \rife{HypLemmeODE} we now have $-V'\ge V^p+bx^\ell$   in $(0,x_1)$,  hence
\be\label{BeqdA3}
{1\over p-1}(V^{1-p})' =-V'V^{-p} \ge 1+bx^\ell V^{-p},\quad 0<x< x_1.
\ee
Since $\frac{(V^{1-p})'}{p-1}\geq 1$, we first deduce that $V\le U_*'$ in $(0,x_1)$ and so 
\be\label{BeqdA4}
V\le U_*',\quad 0<x\le 1/2,
\ee
since $V\le 0$ in $[x_1,1/2]$.
In the case $b=0$, this yields the desired conclusion.

Now assume $b>0$. Combining \eqref{BeqdA3} and \eqref{BeqdA4}, we obtain
$$
{1\over p-1}(V^{1-p})' \ge 1+c\,x^{\ell+\alpha+1},\quad 0<x<x_1.
$$
Consequently,
$$
V^{1-p}\ge (p-1)x[1+c\,x^{\ell+\alpha+1}],\quad 0<x<x_1.$$
Taking a smaller constant $c$ if necessary, we deduce
$$V(x) \le U_*'(x)[1+c\,x^{\ell+\alpha+1}]^{-\alpha}
\le U_*'(x)[1-c\,x^{\ell+\alpha+1}]=U_*'(x)-c\,x^{\ell+1},\quad 0<x<x_1,$$
  and this remains true on the remaining interval $[x_1,1/2]$ where $V\le 0$.  
{The estimate for $u(T,x)$} follows by a further integration. 
\end{proof}

\vskip0.4em

\begin{rem} \label{remb0}
As consequence of the above proof, we also have that if
$u_t(T,x)\ge 0$ in $(0,a]$ for some $T\ge T^*$ and $a\in (0,1/2)$, along with \rife{Beqdd}, then 
$$
u(T,x)\ge u(T,0)+U_*(x)
\quad\hbox{ and }\quad
u_x(T,x)\ge U'_*(x),
\qquad 0<x\le a.
$$
\end{rem}

We stress a {simple but} useful 
consequence of the comparison between $u$ and $U_*$ which occurs if $u_t\leq 0$ at some time $t_0$.

\begin{prop}\label{utneq}  Let $\phi\in X_1$ with $\phi$  
  symmetric and  nondecreasing  on $[0,1/2]$. 
Assume that there exists $t_0>0$ such that 
$$
u_t(t_0,\cdot) \leq 0\quad\hbox{ in $(0,1)$} \qquad\hbox{and}\qquad u(t_0,0)=0.
$$
Then we have
\be\label{compuUstar}
u(t,x)\leq U_*(x) \quad\hbox{ in $[t_0,\infty)\times[0,1]$,}
\ee
\be\label{compuxUstar}
u_x(t,x)\leq U_*'(x) \quad\hbox{ in $[t_0,\infty)\times (0,1)$,}
\ee
\be\label{computUstar}
u_t(t,x) \leq 0 \quad\hbox{ in $[t_0,\infty)\times (0,1)$.}
\ee
\end{prop}

\begin{proof}   Property \eqref{compuUstar}  at $t=t_0$ is true as a consequence of Lemma~\ref{separation}(ii). Since $u_k\leq u$, we also have  $u_k(t_0,x)\leq U_*(x)$ for all $k\geq 1$. Since $U_*$ is  a supersolution 
of problem \rife{app-1d}, by comparison principle we deduce  that $u_k(t,x)\leq U_*(x)$ for all $t>t_0$. Passing to the limit as $k\to \infty$, we get property \eqref{compuUstar} (in particular, $u$ does not lose the boundary condition after $t_0$).
\smallskip

   In order to verify \eqref{computUstar}, fixing any $h>0$,
it suffices to show that 
$$u(t+h,x)\le u(t,x)\quad\hbox{ in $[t_0,\infty)\times [0,1]$.}$$
To this end, fix any $\tau>0$, and set $Q:=(0,\tau)\times (0,1)$ and $Pv:=v_t-v_{xx}-|v_x|^p$.
First applying the comparison principle in Proposition~\ref{compP} with $v_1(t,x)=u(t_0+t,x)$, $v_2(t,x)=u(t_0,x)$ and noting that 
$$Pv_2=-[u_{xx}+|u_x|^p](t_0,x)=-  u_t(t_0,x)\ge 0\quad\hbox{ in Q},$$
and that $v_1=0$ at $x\in \{0,1\}$ by \eqref{compuUstar},
we obtain $u(t_0+t,x)\le u(t_0,x)$ in $Q$.
Next apply the comparison principle with $v_1(t,x)=u(t_0+t+h,x)$, $v_2(t,x)=u(t_0+t,x)$. Since $v_1(0,x)=u(t_0+h,x)\le u(t_0,x)=v_2(0,x)$ by the previous step,
we reach the desired conclusion.  
\smallskip

Finally, \eqref{compuxUstar} follows as a consequence of \eqref{computUstar} and Lemma~\ref{separation}(ii) with $b=0$. 
\end{proof}

\section{Zero-number}
\label{SecZ}

In this section, we continue the study of the one-dimensional case, with $\Omega=(0,1)$,
   turning our attention to the zero-number properties of $u_t$.  
The function $w:=u_t$ is a classical solution of the homogeneous linear parabolic equation
\be\label{eqnut}
w_t-w_{xx}=b(t,x)w_x\quad\hbox{ in $Q:=(0,\infty)\times (0,1)$},
\ee
with coefficient given by $b(t,x):=p|u_x|^{p-2}u_x$.
This allows for using the powerful tool of zero-number.

To this purpose, we will need to consider initial data $\phi \in C^2([0,1])$ which are compatible at order two, that~is:
\be\label{compat2}
\phi=\phi''+|\phi'|^p=0\qquad \hbox{at $x=0$ and $x=1$.}
\ee
This guarantees that $u_t$ is continuous up to $t=0$ and up to the boundary,
i.e. $u_t\in C([0,T^*)\times [0,1])$, with 
\be\label{ut0}
u_t(0,\cdot)=\phi''+|\phi'|^p.
\ee

\begin{defn}
For any function $\psi\in C(0,1)$, we denote by $Z(\psi)$ the number of sign changes of $\psi$ in the interval $(0,1)$. Precisely, we have 
$$Z(\psi)=\sup\bigl\{m\in \N;\ \hbox{there exist $0<x_0<\dots<x_m<1,\ \psi(x_{i-1}) \psi(x_i)<0$, $i=1,\dots,m$}\bigr\}$$
(with the convention $Z(\psi)=0$ if $ \psi$ does not change sign).
In particular, 
if $\phi \in C^2([0,1])$ is compatible at order two, 
we set 
$$N(t)=Z(u_t(t,\cdot)),\quad t\ge 0,$$
where $u$ is the unique viscosity solution of \rife{vhj1}.
\end{defn}

Let us first consider the time range $t\in (0,T^*)$. There, the function $w=u_t$ is a {\it classical solution of \eqref{eqnut} up to the boundary}, and it satisfies the
Dirichlet condition $w(t,0)=w(t,1)=0$. Consequently, the fundamental properties of the zero-number (cf.~\cite{Matano}, \cite{An88}) will be valid.
More generally, for the problem 
\be\label{genapproxpbm}
\begin{cases}
v_t-v_{xx} =F(v_x), & \quad t>0,\ x\in (0,1),\\
v(t,x)=0, & \quad t>0,\ x\in \{0,1\},\\
v(0,x)=\phi(x), & \quad x\in (0,1),
\end{cases}
\ee
with $F\in W^{2,\infty}_{loc}(\R,\R)$, we have the following result as a consequence of \cite{An88}:

\begin{prop}\label{zeronumberut} 
Let $F\in W^{2,\infty}_{loc}(\R,\R)$ and  assume that $\phi\in W^{3,\infty}(0,1)$ satisfies $\phi=\phi''+F(\phi')=0$ for $x\in\{0,1\}$ and $Z(\phi''+F(\phi'))<\infty$.
Let $0<T\le \infty$ and let $v\in C^{1,2}([0,T)\times[0,1])$ be a classical solution of \eqref{genapproxpbm} on $(0,T)$. Then:

\begin{itemize}

\item[(i)] The function $Z(v_t(t))$ is nonincreasing on $[0,T)$.

\item[(ii)] For each $t\in (0,T)$, the zero set $\{x\in [0,1];\ v_t(t,x)=0\}$ is finite.

\item[(iii)] If $v_t$ has a degenerate zero, i.e.~$v_t(t_0,x_0)=v_{tx}(t_0,x_0)=0$, for some $t_0\in (0,T)$ and $x_0\in [0,1]$,
then $Z(v_t(t))$ drops at $t=t_0$, namely:
$$Z(v_t(s))<Z(v_t(t)),\quad 0<s<t_0<t<T.$$

\item[(iv)] For any $\eps\in\R^*$, properties (i)-(iii) remain valid if the function $v_t$ is replaced with $v_t+\eps$.
\end{itemize}
\end{prop}

\begin{rem}
(a) In particular, for $\phi$ as above, properties (i)-(iii) are true for 
$v=u$, $T=T^*$ and $N(t)=Z(u_t(t))$.

(b) Assertion (iv) of Proposition~\ref{zeronumberut} is motivated by the study of the behavior of the solution for $t>T^*$.
Indeed, to our purposes it will be useful to consider the zero-number of the perturbed function $u_t+\eps$ with small $\eps$.
 \end{rem}

\begin{proof}
The function $w=v_t\in C([0,T)\times [0,1])$  
 is a classical solution of
\be\label{genapproxpbm2}
\begin{cases}
w_t-w_{xx} =b(t,x)w_x, & \quad 0<t<T,\ x\in (0,1),\\
w(t,x)=0, & \quad 0\le t<T,\ x\in \{0,1\},\\
\end{cases}
\ee
with coefficient $b(t,x)=F'(v_x)$.
Assertions (i)-(iii) follow from \cite{An88} provided we show that $b, b_x, b_t\in L^\infty((0,\tau)\times (0,1))$ for each $\tau\in (0,T)$.
This is clear for $b$ and, since $b_x=F''(v_x)v_{xx}$ a.e.,  
this is also true for $b_x$. 
Let us consider $b_t=F''(v_x)v_{tx}$ a.e.
Since $w(0,\cdot)\in W^{1,\infty}(0,1)$, 
a standard fixed-point argument, based on the variation of constants formula and heat semigroup estimates,
guarantees that $w_x$, hence $b_t$, is bounded for small $t>0$, hence on $(0,\tau)$.

As for assertion (iv), since $w=v_t+\eps$ is a classical solution of 
$$
\begin{cases}
w_t-w_{xx} =b(t,x)w_x, & \quad 0<t<T,\ x\in (0,1),\\
w(t,x)=\eps, & \quad 0\le t<T,\ x\in \{0,1\},\\
\end{cases}
$$
in view of the above properties of $b$, it follows from Theorem D in \cite{An88}.
\end{proof}

Zero-number properties of $u_t$ are not a priori valid after $T^*$,
especially in case of loss of boundary conditions of $u$, since then $u_t$ is no longer controlled (it does not even need  to exist) on the boundary.
In order to look for some possible zero-number properties of $u_t$ after $T^*$, it is natural to take advantage of the approximation
of $u$ by the global classical solutions $u_k$ of the truncated problems {\rife{app-1d}}.  
Let us thus denote by
$$N_k(t)=Z(u_{k,t}(t,\cdot)),\quad t\ge 0,$$
the number of sign changes of $u_{k,t}(t,.)$.

It is not known whether the monotonicity property of $N(t)$ in Proposition~\ref{zeronumberut} remains true in general for $t>T^*$.
  It will turn out that this is indeed the case when $N(0)\le 4$ (see Proposition~\ref{zeronumbermonotone}),
but this will follow a posteriori as a consequence of the full analysis of the behavior of the global viscosity solution.
In order to carry out this analysis, the following weaker zero-number property will be sufficient for our needs.
It is valid independently of the value of $N(0)$ and can be easily deduced from Proposition~\ref{zeronumberut}.

\begin{prop}\label{zeronumberut2} 
Assume that $\phi\in W^{3,\infty}(0,1)$ is compatible at order two and satisfies $N(0)\equiv Z(\phi''+|\phi'|^p)<\infty$.
Then, for all $t_0\in [0,T^*)$, we have
\be\label{utalternate0A}
N(t)\le N(t_0),\quad t_0<t<\infty.
\ee
In addition, for any $\eps\in\R^*$, we have
\be\label{utalternate0B}
Z(u_t(t)+\eps)\le Z(u_t(t_0)+\eps),\quad t_0<t<\infty.  
\ee
\end{prop}

\begin{proof}
By property (i) of Proposition~\ref{zeronumberut} applied to problem {\rife{app-1d}}, 
we have
$N_k(t)\le N_k(t_0)$ for all $k>\|\phi'\|_\infty$.
By property \eqref{eqabA0}, it follows that 
$$ 
N_k(t)\le N(t_0)\quad\hbox{ for all $k\ge k_0(t_0)$.}
$$ 
Assume for contradiction that $N(t)\ge m:=N(t_0)+1$. Then there exist 
$0<x_0<\dots<x_m<1$ such that 
\be\label{utalternate}
u_t(t,x_{i-1}) u_t(t,x_i)<0,\quad i=1,\dots,m.
\ee
By the approximation property \rife{approxpbm2}, for $k\ge k_0$ large enough, \eqref{utalternate} remains true with $u_t$ replaced with $u_{k,t}$,
hence $N_k(t)\ge m= N(t_0)+1$, which is a contradiction.
This proves \eqref{utalternate0A}.

The proof of \eqref{utalternate0B} is exactly the same, using property (iv) of Proposition~\ref{zeronumberut}.
\end{proof}

In view of Theorems~\ref{nonmin0} and \ref{minimal0}, we now turn to the special case $N(0)=2$, 
for which we will be able to obtain a fairly complete picture of the behavior of solutions.
The condition $N(0)=2$ can be verified to hold for a large class of initial data
satisfying all our requirements, 
  and producing both minimal and nonminimal GBU solutions.
As an example of generic construction of such initial data,   we have the following lemma:

\begin{lem}\label{constuctID}
Let $\vfi\in W^{1,\infty}(0,1)$ be symmetric on $[0,1]$ and satisfy, for some $x_0\in (0,1/2)$:
$$\vfi(0)=0,\quad \textstyle\int_0^{1/2} \vfi(t)dt = 0,$$
$$\hbox{$\vfi(x)>0$ on $(0,x_0)$, \quad $\vfi(x) <0$ and nonincreasing on $(x_0,1/2]$.}$$
(i) Then the function $\phi$ defined by
$$
\phi(x)=
\begin{cases}
\int_0^x \int_0^y \vfi(t)dtdy, &\quad\hbox{$0\le x\le 1/2$,} \\
\noalign{\vskip 2mm}
\phi(1-x),&\quad\hbox{$1/2< x\le 1$,}  
\end{cases}
$$
has the following properties:
\vskip0.3em

\begin{itemize}

\item[(a)] $\phi\in W^{3,\infty}(0,1)$, $\phi$ is symmetric;
\vskip0.1em

\item[(b)] $\phi(0)=\phi'(0)=\phi''(0)=0$ and $\phi$ is in particular compatible at order two;
\vskip0.1em

\item[(c)]  $\phi'(x)>0$ on $(0,1/2)$;
\vskip0.1em

\item[(d)] The function $V:=  \phi''+( \phi')^p$ has a unique zero $z_0$ on $(0,1/2]$ with $V>0$ on $(0,z_0)$ and $V<0$ on $(z_0,1/2]$.

\end{itemize}
\vskip0.3em
\noindent (ii) For any $\lambda >0$, the function $\lambda \phi$ satisfies the same conditions (a)-(d). 
Furthermore, there exists $\lambda^*>0$ such that 
the corresponding solution $u^\lambda$ of \eqref{vhj1} is:

\begin{itemize}

\item[-] a nonminimal GBU solution for $\lambda>\lambda^*$,

\item[-] a minimal GBU solution for $\lambda=\lambda^*$,

\item[-] a global classical solution for $0\le\lambda<\lambda^*$.
\end{itemize}
\end{lem}

\begin{proof}
The function $\phi$ is symmetric by construction and we have
$\phi\in C^2([0,1])$ (the $C^2$ regularity at $x=1/2$ being guaranteed by $\phi'(1/2)=\int_0^{1/2} \vfi(t)dt= 0$).
Since $\phi''=\vfi$ on $[0,1/2]$, properties~(a) and (b) follow.

Property (c) is due to the fact that $\phi'(x)=\int_0^x \vfi(t)dt$ is positive on $(0,x_0)$ and decreasing on $[x_0,1/2)$ with $\phi'(1/2)=0$);

To check property (d) we note that $\phi''=\vfi$, hence
$ \phi''+( \phi')^p$ is positive on $(0,x_0]$ and is  decreasing on $(x_0,1/2]$, with negative value at $x=1/2$).

{For any $\lambda>0$, since $\lambda\varphi$ verifies the same assumptions as $\varphi$,
we deduce that $\lambda \phi$ also satisfies conditions (i)-(iv).}

As for the existence of $\lambda^*>0$ with the stated properties, it follows from \cite{PS2} (see Theorems 2 and 3 and the proof of Theorem 2).
\end{proof}

\vskip0.5em

Our first basic result for the case $N(0)=2$ states that if $T^*<\infty$,
then $N(t)$ remains equal to~$2$ at least until the blow-up time $T^*$.
  We also have a control on the zero-number of the perturbed function $u_t+\eps$,
as well as of time-translates of $u$ itself,
both properties that will be useful in what follows.

\begin{lem}\label{defz}
Let $\phi\in W^{3,\infty}(0,1)$ be compatible at order two,
with $\phi$ symmetric and  nondecreasing on $[0,1/2]$.
Assume that $N(0)=2$ and that $T^*<\infty$.  Then we have:

  (i) $N(t)=2$ for all $t\in (0,T^*)$.
Moreover, $u_t(t,\cdot)$ has  a unique zero $z(t)\in (0,1/2)$ and we have 
\be\label{prelimz0}
u_t(t,\cdot)>0 \ \hbox{ in $(0,z(t))\cup(1-z(t),1)$},\qquad u_t(t,\cdot)<0 \ \hbox{ in $(z(t),1-z(t))$}
\ee
  and
\be\label{prelimz0B}
u_{tx}(t,z(t))<0.
\ee

(ii) Denote $Z_\eps(t)=Z(u_t(t)+\eps)$ for $\eps>0$.
Then for each $t_0\in (0,T^*)$, there exists $\eps_0(t_0)>0$ such that
\be\label{prelimz0C}
Z_\eps(t)\le 2\quad\hbox{  for all $\eps\in (0,\eps_0)$ and all $t\ge t_0$.}
\ee

(iii) For each $t_0\in (0,T^*)$, there exists $\tau_0(t_0)>0$ such that
\be\label{prelimz0D}
Z(u(t+\tau)-u(t))\le 2\quad\hbox{ for all $\tau\in (0, \tau_0)$ and $t\ge t_0$.}
\ee

\end{lem}

\begin{proof}
  (i) Assume for contradiction that $N(t_0)\ne N(0)=2$ for some $t_0\in [0,T^*)$. 
Then $N(t)\le N(t_0)=0$ for all $t\in(t_0,T^*)$ by Proposition~\ref{zeronumberut} and symmetry.
In view of \eqref{utcenter}, it follows that $u_t(t,.)\le 0$ in $[t_0,T^*)\times [0,1]$, hence $u_{tx}(t,0)\leq 0$ for all $t\in[t_0,T^*)$, due to $u_t(t,0)=0$.
Consequently $u_x(t,0)\le u_x(t_0,0)$ for all $t\in[t_0,T^*)$, hence $\sup_{t\in(t_0,T^*)} \|u_x\|_\infty<\infty$ by~\eqref{controlnormux},
contradicting $T^*<\infty$. 

Now, for each $t\in (0,T^*)$, since $N(t)=2$, there exists $z(t)\in(0,1/2)$ such that
\be\label{prelimz}
\hbox{$u_t(t,\cdot)\ge 0$ in $(0,z(t))\cup(1-z(t),1)$ and $u_t(t,\cdot)\le 0$ in $(z(t),1-z(t))$}
\ee
Assume that $u_t(t,x)=0$ for some $t\in (0,T^*)$ and some $x\in(0,1/2)\setminus\{z(t)\}$.
Then $u_{tx}(t,x)=0$ by \eqref{prelimz}, so that $N(t)$ drops at $t=t_0$ by property (iii) in Proposition~\ref{zeronumberut}.
This drop will also occur in case $u_{tx}(t,z(t))=0$.
But this contradicts the fact that $N(t)=2$ for all $t\in (0,T^*)$.
 \smallskip
 
(ii) As a consequence of \eqref{prelimz0} and \eqref{prelimz0B}, there exists $\eps_0>0$ such that 
$Z_\eps(t_0)=2$ for all $\eps\in (0,\eps_0)$.
Applying \eqref{utalternate0B}   
in Proposition~\ref{zeronumberut2}, we deduce \eqref{prelimz0C}.

 \smallskip
(iii) We first claim that there exists $\tau_0>0$ small, such that
\be\label{utalternate0}
Z(u(t_0+\tau)-u(t_0))\le 2\quad\hbox{ for all $\tau\in (0, \tau_0)$}.
\ee
Assume the contrary. Then there exist sequences $\tau_j\to 0$ and
$0<x_{0,j}<\dots<x_{4,j}<1$ such that 
$$
w(\tau_j,x _{i-1,j}) w(\tau_j,x_{i,j})<0,\quad i=1,\dots,4,
$$
where $w(\tau,x)=u(t_0+\tau,x)-u(t_0,x)$. By applying the mean value theorem, we deduce the existence of sequences
$0<y_{1,j}<\dots<y_{4,j}<1$ such that $w(\tau_j,y _{i,j})=0$, and then of sequences $t_{i,j}\to 0$ such that
$$u_t(t_0+t_{i,j},y_{i,j})=0,\quad i=1,\dots,4.
$$
Passing to the limit, up to subsequences, we may assume $y_{i,j}\to y_i\in [0,1]$.
Consequently, we have $u(t_0,y_i)=0$, for $i=1,\dots,4$, with either $0<y_1<\dots<y_4<1$ or $u_{tx}(t_0,y_i)=0$ for some $i$.
In each case, this is a contradiction with assertion (i).

Consequently, \eqref{utalternate0} is true. Since $w(t,x)=u_k(t+\tau,x)-u_k(t,x)$ is a classical solution of \eqref{genapproxpbm2} with
$$b(t,x) 
=\int_0^1 F'_{k}\bigl[u_{k,x}(t,x)+s(u_{k,x}(t+\tau,x)-u_{k,x}(t,x))\bigr]\, ds,$$
it follows from \cite{An88}, similarly as in the proof of Proposition~\ref{zeronumberut}, that $N(u_k(t+\tau)-u_k(t))\le N(u_k(t_0+\tau)-u_k(t_0))\le 2$ for $t\ge t_0$.
We finally deduce property \eqref{prelimz0D} by passing to the limit $k\to\infty$ exactly as in the proof of \eqref{utalternate0A}.
\end{proof}

The following simple property enables one to reduce the case $N(0)=4$ to the case $N(0)=2$, after a time shift. 

\begin{lem}\label{higherN}
Let $\phi\in W^{3,\infty}(0,1)$ be compatible at order two,
with $\phi$ symmetric and  nondecreasing on $[0,1/2]$. 
Assume $N(0)\equiv Z(\phi''+|\phi'|^p)<\infty$ and $T^*(\phi)<\infty$. Then:

(i) There exist $t_0\in (0,T^*)$ and an integer $N_*\le N(0)$, with $N_*\equiv 2$ \hbox{\rm [mod. $4$]}, such that 
\be\label{propNstar}
N(t)=N_*\quad\hbox{ for all $t\in [t_0,T^*)$.}
\ee

(ii) In particular, if $N(0)=4$, then $N_*=2$.
\end{lem}

\begin{proof}
Since $N(t)$ is integer-valued and nonincreasing for $t\in (0,T^*)$, by Proposition~\ref{zeronumberut}(i), the existence of $t_0$ and 
$N_*\le N(0)$ satisfying \rife{propNstar} follows
($N_*$ is given by $\lim_{t\to T^*}N(t)$).

By symmetry, $N_*$ is even. Assume for contradiction that $N_*\equiv 0$ \hbox{\rm [mod. $4$].}
Then, for all $t\in [t_0,T^*)$, since $u_t(t,1/2)<0$ by \eqref{utcenter}, we have $u_t(t,x)\le 0$ for $x>0$ small.
Arguing as in the first paragraph of the proof of Lemma~\ref{defz}, we reach a contradiction with $T^*<\infty$.
\end{proof}

\begin{rem} \label{remhigherN}
All the conclusions of Theorems~\ref{nonmin0} and \ref{minimal0} remain valid if assumption \eqref{hypID2}
 is replaced with $N(0)\equiv Z(\phi''+|\phi'|^p)=4$.
Indeed, by Lemma~\ref{higherN}, there exists $t_0\in (0,T^*)$ such that $u(t_0,\cdot)$, considered as new initial data (at the new time origin $t_0$),
satisfies all the assumptions of Theorems~\ref{nonmin0} and \ref{minimal0}.
\end{rem}

Let us go back to the case $N(0)=2$ (under the assumptions of Lemma~\ref{defz}). 
The limiting behavior of the (unique) zero point $z(t)$ of the time derivative will play a crucial role in the behavior of the solution at $T^*$. 
  Indeed, the key quantity to consider will turn out to be 
\be\label{defL}
L:=\liminf_{t\to T^*_-}z(t).
\ee
This kind of analysis will be carried out hereafter. For convenience, we split the analysis between the case of  minimal and nonminimal blow-up solutions.

\section{Behavior of nonminimal solutions.    Proofs of Theorems~\ref{nonmin0} and \ref{proppersist}} \label{behavenonmin}
\label{Sec-nonmin}

\subsection{Behavior near the blow-up time $T^*$.}
\label{Sec-nonmin1}

We shall here prove the following result.
It shows  that blow-up is nonminimal if and only if $L>0$, where $L$ is defined in \eqref{defL} 
and $z(t)$ is the unique zero point of $u_t$ in $(0,1/2]$ (cf.~Lemma~\ref{defz}).
For nonminimal blow-up solutions, it also implies part of assertion (i) of Theorem~\ref{nonmin0} 
and the GBU rate estimate \eqref{estA0} therein.
{In addition, it gives some first linear bounds on the rate of loss of boundary conditions, that will be later improved to the precise asymptotics \eqref{estB0}. }

\begin{prop}\label{nonmin} 
Let $\phi\in W^{3,\infty}(0,1)$ be compatible at order two,
with $\phi$ symmetric and  nondecreasing on $[0,1/2]$.
Assume that $N(0)=2$ and $T^*<\infty$.
Then we have 

(i) $u$ is a nonminimal blowup solution if and only if $L>0$.
Moreover, if $u$ is a minimal solution, then
\be\label{utneg}
u_t(T^*,x)\le 0,\quad x\in (0,1).
\ee

  (ii) If $L>0$, then  there is immediate loss of boundary conditions (LBC) for $t>T^*$
and $u$ satisfies the following estimates
\be\label{estA}
c_1(T^*-t)^{-1/(p-2)}\le \|u_x(t)\|_\infty \le c_2(T^*-t)^{-1/(p-2)},\quad\hbox{as $t\to T^*_-$,}
\ee
and
\be\label{estB}
c_3(t-T^*)\le u(t,0) \le c_4(t-T^*),\quad\hbox{as $t\to T^*_+$,}
\ee
for some constants $c_i>0$.   Moreover, there exist $c,a,\delta>0$ such that
\be\label{estB2}
u_t(t,x) \ge c \quad\hbox{ for all $t\in (T^*,T^* +\delta)$ and $x\in (0,a]$.}
\ee
\end{prop}

\begin{rem} Property \eqref{estB2} shows that, interestingly, $u_t$ is uniformly {positive} near the corner $(T^*_+,0)$ and undergoes a jump discontinuity there,
while remaining bounded. A similar property will be observed near the regularization time (see~\eqref{estB2r}).
\end{rem}

\begin{proof}[Proof of Proposition~\ref{nonmin}]
  First assume $L=0$. Since $u_t$ is continuous in $(0,T^*]\times (0,1)$, property \eqref{utneg} follows.
Proposition~\ref{utneq} then guarantees that the boundary conditions are preserved forever. So $u$ is minimal on account of Proposition~\ref{prop-min}. 
\smallskip

We assume henceforth that $L>0$. We aim at showing that 
  $u$ loses the boundary conditions immediately after $T^*$
(hence in particular $u$ is nonminimal) and satisfies estimates \eqref{estA}, \eqref{estB}. 
We divide the proof in several steps. {We consider below the approximating sequence of solutions $u_k$ of \rife{app-1d}, with   $F_k$ given by \rife{eqaaa}}.

\smallskip

{\bf Step 1.} {\it Positivity of $u_{k,t}$ in a   uniform  neighborhood of $t=T^*$ and $x=0$.} 
  We first claim that
\be\label{eqposut}
u_t(T^*,x)>0\quad\hbox{ in $(0,L)$.}
\ee
Fix $0<L'<L$. By definition of $L$, there exists $t_0=t_0(L')$ such that $u_t(t,x)>0$ in $[t_0,T^*)\times (0,L')$.
For given $a\in (0,L')$, $v:=u_t$ is a classical solution of $v_t-v_{xx}=b(t,x)v_x$ in $Q_a:=[t_0,T^*]\times[a,L']$,
where the coefficient $b(t,x):=pu_x^{p-1}$ is bounded in $Q_a$.
The strong maximum principle then guarantees that $v(T^*,x)>0$ in $(a,L')$ and \eqref{eqposut} follows.

\smallskip

Set $a=L/2$. By \eqref{eqposut} and the local convergence of the $u_k$, there exist $k_0\ge 1$ and $\eta,\delta>0$ such that,
 for all $k\ge k_0$, 
\be\label{eqba}
u_{k,t}(t,a)\ge \eta \quad\hbox{ for all $t\in [T^* -\delta,T^* +\delta]$.}
\ee
By property (i) of Proposition~\ref{zeronumberut} applied to problem \rife{app-1d},
we have
$N_k(t)\le N_k(0)=N(0)=2$ for all $k>\|\phi'\|_\infty$.
Therefore, \rife{eqba} and \eqref{uktcenter} guarantee that  
\be\label{Nkt2}
N_k(t)=2\quad\hbox{ for all $t\in [0,T^* +\delta]$,}
\ee
and we thus have
\be\label{eqb}
u_{k,t}(t,x)>0 \quad\hbox{   in $[T^* -\delta,T^* +\delta]\times(0,a]$ for all $k\ge k_0$. }
\ee

{\bf Step 2.} {\it 
Maximum principle   for $u_{k,x}$.}  {Let us set
$$
m_k(t):=\max_{\overline Q_t} |u_{k,x}|\,.
$$}
We claim that there exists $t_0\in (T^*-\delta,T^*)$ such that, for all sufficiently large $k$,
\be\label{eqbak}
m_k(t)=u_{k,x}(t,0) \ge 1,\quad t_0\le t\le T^* +\delta.
\ee
Since $w=u_{k,x}$ satisfies the equation 
$$w_t-w_{xx}=F'_k(u_{k,x})w_x$$
we have
$$
m_k(t)=\max\Bigl[\|\phi_x\|_\infty,\max_{\tau\in [0,t]} u_{k,x}(\tau,0)\Bigr].
$$
By \rife{Nkt2} and \eqref{uktcenter}, we see that  the function $t\mapsto u_{k,x}(t,0)$ is nondecreasing on $[0,T^* +\delta]$, hence
$$
m_k(t)=\max\Bigl[\|\phi_x\|_\infty,u_{k,x}(t,0)\Bigr],\quad t\in [0,T^* +\delta].
$$
Since $\|u_x(t)\|_\infty\to\infty$ as $t\to T^*$, there exists $t_0\in (T^*-\delta,T^*)$ such that $\|u_x(t_0)\|_\infty>1+\|\phi_x\|_\infty$,
so that $\|u_{k,x}(t_0)\|_\infty\ge 1+\|\phi_x\|_\infty$ for all $k\ge k_0$ large enough, hence \eqref{eqbak}.
\medskip

We proceed by deriving an upper bound for $m_k'$. 

\begin{lem}   Under the assumptions of Proposition~\ref{nonmin},
there exists a constant $C>0$ independent of $k$ such that, for all large enough $k$,  
\be\label{eqca}
0\le  m'_k(t)\le CF'_k(m_k(t)),\quad t_0\le t\le T^* +\delta
\ee
and
\be\label{eqcaa}
0\le  m'(t)\le CF'(m(t)),\quad t_0\le t<T^*.
\ee
\end{lem}

\begin{proof}
We use (a simpler version of) arguments from the proofs of Theorem~\ref{minimal_rateG}
and Lemma~\ref{unbddT}. 
 Set $w:=u_{k,t}$, $F=F_k$.
Let $t\in (t_0,T^* +\delta]$, $s\in (0,t)$, and put
$K=\sup_{\sigma\in [0,t-s]} \sigma^{1/2}\|w_x(s+\sigma)\|_\infty$.
For $\tau\in (0,t-s)$, in view of  the variation-of-constants formula, we have
$$w(s+\tau)=e^{\tau \Delta}w(s)+
\int_0^\tau e^{(\tau-\sigma)\Delta}(F'(u_{k,x})w_x)(s+\sigma)\, d\sigma.$$
Since
$\int_0^\tau (\tau-\sigma)^{-1/2}\sigma^{-1/2}\, d\sigma
=\int_0^1 (1-z)^{-1/2}z^{-1/2}\, dz$,
it follows that
\begin{align*}
\|w_x(s+\tau)\|_\infty
&\leq C\tau^{-1/2}\|w(s)\|_\infty+
C\int_0^\tau (\tau-\sigma)^{-1/2}\|[F'(u_{k,x})w_x](s+\sigma)\|_\infty\, d\sigma\\
&\leq C\tau^{-1/2}+CF'(m_k(t))K,
\end{align*}
where we also used the convexity of $F_k$. {Here and below $C$ denotes possibly different constants independent of $k$.}
Multiplying by $\tau^{1/2}$ and taking the supremum for $\tau\in [0,t-s]$, we obtain
$$K\leq C+C(t-s)^{1/2}F'(m_k(t))K.
$$
Now choosing $s=t-(1/4)\min\bigl(t_0,[CF'(m_k(t))]^{-2}\bigr)\in (0,t)$, we obtain
$K\leq 2C$, hence
$$\|w_x(t)\|_\infty\leq 2C(t-s)^{-1/2}\leq 4C\max(t_0^{-1/2},CF'(m_k(t)))\leq CF'(m_k(t))$$
(recalling $m_k(t)\ge 1$).
Since, by \rife{eqb} and  \eqref{eqbak},  $0\le  m'_k(t)=u_{k,xt}(t,0)\le \|w_x(t)\|_\infty$ for $t_0\le t\le T^* +\delta$,
estimate \rife{eqca} follows.

Finally, for each $t\in [t_0,T^*)$, we have 
\be\label{eqcaamt}
m(t)=\lim_{k\to\infty}m_k(t)=\lim_{k\to\infty}u_{k,x}(t,0)=u_{x}(t,0)
\ee
 (passing to the limit in \eqref{eqbak}),
and $m'(t)=u_{xt}(t,0)=\lim_{k\to\infty}u_{k,xt}(t,0)=\lim_{k\to\infty}m'_k(t)$.
Estimate \rife{eqcaa} then follows by letting $k\to\infty$ in~\rife{eqca}.
\end{proof}

{\bf Step 3.} {\it    Auxiliary function argument  applied to the $u_k$.}

\begin{lem} \label{lem73}
{  Under the assumptions of Proposition~\ref{nonmin}, } we have
$$u_{k,t} \ge C\Bigl(1-{u_{k,x} \over m_k(t)}\Bigr) \quad\hbox{ for all $t_0\le t\le T^* +\delta$ and $0<x\le a,$}$$
with $C$ independent of $k$.
\end{lem}

\begin{proof}
  The proof is based on a modification of a device from \cite{GH08},
based on the auxiliary function in \rife{eqc2}. 
Fix $k$ and denote $v:=u_k$, $F=F_k$,   $\mu=m_k$ for simplicity. {We denote by $C$ possibly different constants independent of $k$. }
Consider the parabolic operator
$${\mathcal L}h:=h_t-h_{xx}-F'(v_x)h_x$$
in $Q:=[t_0,T^* +\delta]\times[0,a]$.  In view of \rife{eqaa} and $v_x\ge 0$ in $Q$, we note that
\be\label{eqc}
{\mathcal L}v=F(v_x)-F'(v_x)v_x\le -F(v_x).
\ee
For $\sigma\in(0,1)$ to be chosen later, 
we introduce the auxiliary function
\be\label{eqc2}
w(t,x):=\Bigl(1+{1\over \mu^\sigma(t)}\Bigr)\Bigl(1-{v_x\over \mu(t)}\Bigr).
\ee
A direct computation using ${\mathcal L}v_x=0$ shows that
$$
{\mathcal L}w=-{\sigma \mu'\over \mu^{\sigma+1}}\Bigl(1-{v_x\over \mu}\Bigr)
+\Bigl(1+{1\over \mu^\sigma}\Bigr){v_x \mu'\over \mu^2}.
$$

Now consider the following cases separately:
\medskip

\noindent {\it Case 1.} If $v_x(t,x)\le{\sigma\over \sigma +2}\mu^{1-\sigma}(t)$, then 
since $\mu'(t), v_x\geq 0$ and $\mu(t)\ge 1$, we have
$${\mathcal L}w
={\mu'\over \mu^{\sigma+1}} \Bigl(-\sigma+(\sigma+1){v_x\over \mu} 
+{v_x \over \mu^{1-\sigma}}\Bigr)
\leq{\mu'\over \mu^{\sigma+1}} \Bigl(-\sigma+(\sigma+2)
{v_x \over \mu^{1-\sigma}} \Bigr)\leq 0.$$

\medskip
\noindent {\it Case 2.} If $v_x(t,x)>{\sigma\over \sigma +2}\mu^{1-\sigma}(t)$, then, 
by \rife{eqca} and \rife{eqaa}, we have
\be\label{eqd}{\mathcal L}w
\leq\Bigl(1+{1\over \mu^\sigma}\Bigr){v_x \mu'\over \mu^2}
\leq Cv_x{F'(\mu)\over \mu^2}
\leq Cv_x{F(\mu)\over \mu^3},
\ee
and we have the following two subcases.

\vskip0.3em
\  {\it Case 2a.} If $v_x(t,x)\le k$, then, choosing $\sigma\in (0,1)$ such that $\sigma\le 2/(p-1)$, 
hence $(1-\sigma)(p-1)\ge p-3$, we have
$${F(v_x)\over v_x}=v_x^{p-1}\ge C\mu^{(1-\sigma)(p-1)}\ge C\mu^{p-3}\ge C{F(\mu)\over \mu^3},$$
due to \rife{eqaaa} and $\mu(t)\ge 1$. This along with \rife{eqd} yields
\be\label{eqe}{\mathcal L}w\le CF(v_x).
\ee

\medskip
\  {\it Case 2b.} If $v_x(t,x)>k$, hence $\mu(t)>k$, then 
$${F(v_x)\over v_x}\ge {C\, k^{p-1} \ge C\, {F(\mu)\over \mu^2}}\ge C{F(\mu)\over \mu^3}$$
by \rife{eqaaa} and $\mu(t)\ge 1$, hence again \rife{eqe}.
\medskip

In all cases we thus have \rife{eqe}, hence
$${\mathcal L}(w+Cv)
\leq 0={\mathcal L}v_t \quad\hbox{ in $Q$},$$
due to \rife{eqc}.
Moreover, we have
$$[w+Cv](t,0)=0=v_t(t,0),\quad t_0\leq t\le T^* +\delta, 
$$
{where we used that $v_x(t,0)=\mu(t)$, see \rife{eqbak}.}
On the other hand, by \eqref{eqabA0},  we have $u_k\equiv u$ on $Q_{t_0}$ for all $k$ large enough.
Also, as a consequence of \rife{eqb}, we have $u_t(t_0,x)>0$ in $(0,a]$, as well as $u_{tx}(t_0,0)>0$, in view of the Hopf lemma. 
Consequently there exists $K>0$ (independent of $k$) such that
$$\bigl[(w+Cv)-Kv_t\bigr](t_0,x)\leq 0,\quad  0\le x\le a. $$
Since $v\le \|\phi\|_\infty$ and $w\le 2$, by increasing the constant $K$ if necessary, \rife{eqba} implies
$w+Cv\le Kv_t$ for $x=a$ and $t\in (t_0,T^* +\delta]$.
By the maximum principle, we deduce $w+Cv\leq K v_t$ 
in $[t_0,T^* +\delta]\times(0,a)$, hence the conclusion of Lemma~\ref{lem73}.
\end{proof}

\vskip0.5em

{\bf Step 4.} {\it  LBC and detachment estimates at $t=T^*_+$.}
For $x\in (0,a]$ and $t_0<s<T^*<t<T^* +\delta$ fixed, 
we deduce from Lemma~\ref{lem73} and $m'_k\ge 0$ (cf.~\rife{eqca}) that 
$$u_{k,t}(t,x) \ge C\Bigl(1-{u_{k,x}(t,x) \over m_k(s)}\Bigr).$$
Letting $k\to\infty$, we obtain
$$u_t(t,x) \ge C\Bigl(1-{u_x(t,x) \over m(s)}\Bigr)$$
and then, letting $s\to T^*$, hence $m(s)\to\infty$, we get
$$u_t(t,x) \ge C \quad\hbox{ for all $x\in (0,a]$ and $T^*<t<T^* +\delta,$}$$
  that is, \eqref{estB2}. 
Consequently,
$$u(t,x) \ge C(t-T^*) \quad\hbox{ for all $x\in (0,a]$ and $T^*<t<T^* +\delta.$}$$
Finally passing to the limit $x\to 0$, using the fact that $u\in C([0,\infty)\times [0,1])$, we obtain 
$$u(t,0) \ge C(t-T^*)  \quad\hbox{ for all $T^*<t<T^* +\delta$.}$$

The upper estimate in \rife{estB} is an immediate consequence of the bound on $u_t$.
\bigskip

{\bf Step 5.} {\it  Completion of proof   of Proposition~\ref{nonmin}:  GBU estimates at $t=T^*_-$.}
{As a consequence of Lemma \ref{lem73} and the convergence of $u_k$ to $u$ for $t<T^*$, } we have
$$u_t \ge C\Bigl(1-{u_x \over m(t)}\Bigr) \quad\hbox{ for all $t_0\le t<T^*$ and $0<x\le a.$}$$
Then, following \cite{GH08}, we write, using $m(t)=u_x(t,0)$ (cf.~\rife{eqcaamt}): 
$$
u_{xt}(t,0)
=\lim_{x\to 0+}{u_t(t,x)\over x}
\geq C\lim_{x\to 0+} \frac 1x \Bigl(1-{u_x(t,x) \over m(t)}\Bigr)
= -{Cu_{xx}(t,0)\over m(t)}  \geq {Cu_x^p(t,0)\over m(t)}=Cu_x^{p-1}(t,0),
$$
and the upper estimate in \rife{estA} follows by integration.
As for the lower estimate, it was proved in Theorem~\ref{minimal_rate} (alternatively, in the current simpler situation, it also follows by integrating \rife{eqcaa}).  
\end{proof}

\subsection{Behavior near the reconnection time.}
\label{Sec-nonmin2}

We can now investigate what happens after the blow-up time of nonminimal blow-up solutions. To this extent, we define
$$
T_0:=\inf\{t>T^*;\, u(t,0)=0\} 
$$
the reconnection time, or (first) time of recovery of boundary conditions. 
We already know that $T^*<T_0\le T^r<\infty$.

We shall prove the following result, which
shows that $u$ is regularized forever immediately after~$T_0$, i.e., $T^r=T_0$.
As a consequence, along with Proposition~\ref{nonmin}, this will imply assertion (i) of Theorem~\ref{nonmin0}.
In addition, it establishes the  regularization rate \eqref{BeqConclA0} in Theorem~\ref{nonmin0}
and gives a first linear bound on the rate of recovery of boundary conditions, that will be later improved {to the precise asymptotics} \eqref{BeqConclAa0}.

\begin{prop}\label{recon} 
Let $\phi\in W^{3,\infty}(0,1)$ be compatible at order two,
with $\phi$ symmetric and  nondecreasing on $[0,1/2]$.
Assume that $N(0)=2$, 
that $T^*<\infty$ and that $u$ is a nonminimal blowup solution.
Then $u$ is a classical solution of problem \rife{vhj1} (including boundary conditions)
for all $t\in(T_0,\infty)$.    In other words, $T^r=T_0$. 

Moreover $u$ satisfies the following reconnection and regularization estimates:
\be\label{BeqConclAa}
c_1(T^r-t)\le u(t,0) \le c_2(T^r-t),
\quad\hbox{as $t\to T^r_-$,}
\ee
\be\label{BeqConclA}
c_3(t-T^r)^{-1/(p-2)} \le \|u_x(t)\|_\infty\le c_4(t-T^r)^{-1/(p-2)},
\quad\hbox{as $t\to T^r_+$,}
\ee
with some constants $c_i>0$.
    Furthermore there exist $c>0$ and $t_1\in (T^*,T^r)$ such that
\be\label{estB2r}
u_t\le -c\quad\hbox{ in $[t_1,T^r)\times (0,1)$.}
\ee
\end{prop}

Proposition~\ref{recon} will be shown through a series of lemmas.
\medskip

The first step is to initialize the uniform negativity of $u_t$.
Intuitively, since $u(t,0)$ goes from positive values for $t<T_0$ to $0$ for $t=T_0$, 
the time derivative $u_t$   (although not necessarily defined at $x=0$)  must be negative near $x=0$ for ''many'' values of~$t$.
However, this has to be extended to the whole interval.
The proof relies on a key zero-number argument
  applied to the perturbed function $u_t+\eps$ with $\eps>0$ small. 

\begin{lem}\label{neg1}
Under the assumptions of Proposition~\ref{recon}, 
for any $t_0\in [T^*,T_0)$, there exist $t_1\in (t_0,T_0)$ and $\eps>0$ such that 
$$
u_t(t_1,x)\le -\eps,\quad 0<x<1.
$$
\end{lem}

\begin{proof} 
  Fix any $\tau, t_0$ such that $0<\tau<T^*<t_0<T_0$. 
By \eqref{utcenter}, we have  
\be\label{BoundSigma}
\sigma:=-\max_{t\in [\tau,T_0]} u_t(t,1/2)>0
\ee
and, by Lemma~\ref{defz}(ii),  there exists $\eps_0(\tau)\in (0,\sigma)$ such that
\be\label{prelimz0C2}
Z_\eps(t):=Z(u_t(t)+\eps)\le 2\quad\hbox{  for all $\eps\in (0,\eps_0)$ and all $t\ge \tau$.}
\ee

We claim that, for $\eps\in (0,\eps_0)$ to be determined, 
\be\label{ConclEpsB}
\hbox{there exists $t_1\in (t_0,T_0)$ such that $Z_\eps(t_1)=0$.}
\ee
Assume for contradiction that this is not the case. 
Then $Z_\eps(t)=2$ for all $t\in [t_0,T_0)$, due to \eqref{prelimz0C2}.  
Since $u_t(t,x)+\eps<0$ near $x=1/2$ by \rife{BoundSigma}, we may set 
$$
\xi_\eps(t)=\min\bigl\{x<1/2;\, u_t(t,x)+\eps\le 0\ \hbox{ on $[x,1/2]$} \bigr\}\in (0,1/2).
$$
We note that the function $\xi_\eps(t)$ is l.s.c. on $[t_0,T_0)$.
Indeed for every $t\in[t_0,T_0)$ and every $\eta>0$ there exists 
$x_\eta\in (\xi_\eps(t)-\eta,\xi_\eps(t))$ such that 
$u_t(t,x_\eta)+\eps>0$, hence $u_t(s,x_\eta)+\eps>0$ for all $s$ close enough to $t$,
due to $u_t\in C((0,\infty)\times(0,1))$, so that $\xi_\eps(s)\ge \xi_\eps(t)-\eta$.
It follows that, for any $t\in (t_0,T_0)$,
$$
x_0(t):=\inf_{s\in [t_0,t]}\xi_\eps(s)>0.
$$
Then,   fixing any $t\in (t_0,T_0)$ and $x\in (0,x_0(t))$,  we have $u_t(s,x)+\eps\ge 0$ for all $s\in [t_0,t]$ due to 
$Z_\eps(t)=2$. By integration, we obtain
$$
u(t_0,x)-u(t,x)=-\int_{t_0}^t u_t(s,x)\, ds\le \eps(t-t_0).
$$
Letting $x\to 0$, we get $u(t_0,0)-u(t,0)\le \eps(t-t_0)$, and then, letting $t\to T_0$, 
$$
0<u(t_0,0)\le \eps(T_0-t_0).
$$
Therefore, for any choice $\eps<\min(\eps_0,(T_0-t_0)^{-1}u(t_0,0))$,
property \rife{ConclEpsB} is necessarily true.
Since $u_t(t_1,x)+\eps<0$ near $x=1/2$, we conclude that $u_t(t_1,x)+\eps\leq 0$ in $(0,1)$.
\end{proof}

The second step is to propagate this uniform negativity.
This is done by the maximum principle applied to a suitable family of auxiliary functions
on subdomains (where the $u_x$ is bounded).

\begin{lem}\label{neg2} Under the assumptions of Proposition~\ref{recon}, 
there exist $\eta>0$ and $t_1\in (T^*,T_0)$ such that
\be\label{ConclEta}
u_t\le -\eta<0\quad\hbox{ in $[t_1,T_0)\times (0,1)$}.
\ee
\end{lem}

\begin{proof}
Fix $a\in (0,1/2)$ and consider the auxiliary function
\be\label{defJa}
J(t,x)=u_t-\eta\bigl[La^\alpha (u_x+K)-1\bigr],\quad (t,x)\in Q:=[t_1,T_0)\times [a,1/2],
\ee
where $\eta,   K,  L>0$ and $t_1\in (T^*,T_0)$ are to be chosen.
An immediate computation shows that
$$
{\mathcal L}J:=J_t-J_{xx}-p(u_x)^{p-1}J_x=0\quad\hbox{ in $Q$}.
$$
  Since $u(t,0)>0$ for all $t\in (T^*,T_0)$ (by Proposition~\ref{nonmin} and the definition of $T_0$)
it follows from Lemmas~\ref{bdl} and \ref{basic-prop} that
$u_x(t,a)+K\ge d_pa^{-\alpha}$, where $d_p=(p-1)^{-1/(p-1)}$ and $K$ is independent of $t$.
We choose this $K$ in \eqref{defJa}. 
Since also $u_t\le M$ by \rife{u_t-bounded}, we have
$$
J(t,a)\le M-\eta\bigl(Ld_p-1\bigr).
$$
We also have
$$
J(t,1/2)\le 
-\sigma+\eta,
$$
by \rife{BoundSigma}.
Moreover, with $t_1,\eps$ provided by Lemma~\ref{neg1} (applied with $t_0=T^*$), we have
$$
J(t_1,x)\le -\eps+\eta.
$$
We then make the following choice:
$$
\eta=\min(\sigma, \eps),\ \quad L=d_p^{-1}\bigl(1+M\eta^{-1}\bigr)
$$
(which is independent of $a$).
In view of the above, it then follows from the maximum principle that
$J(t,x)\le 0$ in $(t_1,T_0)\times (a,1/2)$. Consequently, we have
$$
u_t(t,x)\le \eta\bigl[La^\alpha (u_x(t,x)+K)-1\bigr],\quad t_1\le t<T_0,\ 0<a<x\le 1/2.
$$
For any given $(t,x)\in [t_1,T_0)\times (0,1/2)$, we may let $a\to 0^+$ in the
last inequality and the conclusion follows in $(0,1/2]$, 
hence in $(0,1)$ by symmetry.
\end{proof}

We now have all the ingredients for the proof of Proposition~\ref{recon}. 
\medskip

\begin{proof}[Proof of Proposition~\ref{recon}.]
Since  $u\in C^{1,2}((0,\infty)\times(-1,1))$, it follows from \rife{ConclEta}
that \rife{HypLemmeODE} is true at $T=T_0$ with $\ell=0$
and $b=\eta>0$. 
By Lemma~\ref{separation}(ii), we deduce that \rife{BeqabAAa} is true with $m=2$.
Lemma~\ref{barrier} then guarantees that $T^r=T_0$ and 
that the upper estimate in \rife{BeqConclA} is satisfied.
As for the lower estimate in \rife{BeqConclA}, it  was proved in Theorem~\ref{minimal_rate2}.
\smallskip

On the other hand, \eqref{estB2r} and the lower part of estimate \rife{BeqConclAa} are direct consequences of Lemma~\ref{neg2}, 
whereas  the {upper} part is just due to the bound $|u_t|\le M$.
\end{proof}

\subsection{On the life of blow-up solutions between $T^*$ and $T^r$.}
\label{Sec-nonmin3}

This subsection investigates what happens to the nonminimal blow-up solution after the loss of boundary conditions 
and before they are recovered, namely in the time interval between $T^*$ and $T^r$. 
\smallskip

  We shall first prove Theorem~\ref{proppersist}, which is
valid for any LBC solution in one space dimension.
Assertion (i) is a direct consequence of results from Section~\ref{prelim1d}.

\begin{proof}[Proof of Theorem~\ref{proppersist}(i)]
For all $t\in (T_1,T_2)$, since $u(t,0)>0$,
we know from Lemma~\ref{bdl} that $\lim_{x \to 0_+} u_x(t,x) =\infty$
 and estimates \rife{shiftedcopy0}-\rife{shiftedcopy} follow from Lemma~\ref{basic-prop}.
 The estimates at $t=T_1$ and $T_2$ follow by continuity since $u\in C^{1,2}((0,\infty)\times (0,1))\cap C([0,\infty)\times [0,1])$.
\end{proof}

In view of the proof of assertion (ii), which is a bit delicate,
we establish the following bounds on $u_{xt}$ and $u_{tt}$ 
which are the key to the boundary regularity of $u_t$.

\begin{lem} \label{bounduxt}
Under the assumptions of Theorem~\ref{proppersist}, for each $\eta>0$, we have 
\be\label{bounduxt1}
\sup_{(T_1+\eta,T_2)\times (0,1/2)} |u_{xt}|<\infty. 
\ee
\end{lem}

\begin{lem} \label{boundutt}
Let $\phi\in X_1$ with $T^*(\phi)<\infty$. Then, for all $t_0>0$, we have
$$\inf_{(t_0,\infty)\times (0,1)} u_{tt}>-\infty.$$ 
\end{lem}

The key idea for the proof of Lemma~\ref{bounduxt} is to consider the finite differences of $u_x$ in time in the interval $(T_1,T_2)$ and to:

-{\hskip 3pt}observe that they vanish at $x=0$ due to the proximity with the reference profile $U_*'$ (cf.~\rife{shiftedcopy});

-{\hskip 3pt}show that they satisfy a parabolic inequality with {\it strong nonlinear absorption},
which forces their uniform boundedess for any $t>T_1$.

\begin{proof}[Proof of Lemma~\ref{bounduxt}]
{\bf Step 1.} {\it Parabolic equation for finite differences of $u_x$ in time.}
By \rife{shiftedcopy}, 
there exists $a\in (0,1/2)$ such that 
\be\label{bounduxt1aux1}
u_x(t,x)\ge cx^{-\alpha}\quad\hbox{ in $[T_1,T_2]\times (0,a)$,} 
\ee
and
\be\label{bounduxt1aux2}
u_{xx}=-(u_x)^p+u_t\le -cx^{-\alpha-1}\quad\hbox{ in $[T_1,T_2]\times (0,a]$,}
\ee
with $\alpha=1/(p-1)$ and $c=c(p)>0$.

Next fix $h>0$ and let $w=\tau_hu_x$
where $\tau_h$ denotes the finite difference (in time) operator i.e.,
$$[\tau_h\phi](t,x):=h^{-1}(\phi(t+h,x)-\phi(t,x)).$$
In $Q:=[T_1,T_2-h]\times (0,a]$, the function $w$ satisfies 
$$
\begin{aligned}
w_t-w_{xx}
&=p\tau_h[(u_x)^{p-1}u_{xx}]\\
&=p(u_x)^{p-1}(t+h,x)\,\tau_h u_{xx}+pu_{xx}(t,x)\,\tau_h [(u_x)^{p-1}].
\end{aligned}
$$
By the mean value theorem, there exists $\theta=\theta(t,x,h)\in (0,1)$ such that
$$
\tau_h [(u_x)^{p-1}]=(u_x)^{p-1}(t+h,x)-(u_x)^{p-1}(t,x)=(p-1)\bigl[\theta u_x(t+h,x)+(1-\theta)u_x(t,x)\bigr]^{p-2}w(t,x).
$$
It then follows from \rife{bounduxt1aux1}, \rife{bounduxt1aux2} that $w$ satisfies in $Q$ an equation of the form
\be\label{bounduxt1aux2a}
w_t-w_{xx}=-A(t,x)w+B(t,x)w_x,
\ee
with
\be\label{bounduxt1aux2b}
A(t,x)\ge c(p)x^{-2}.
\ee

{\bf Step 2.} {\it Control of $u_{xt}$ away from the boundary.}
We next claim that 
\be\label{bounduxt1claim}
|u_{xt}|\le Cx^{-3-\alpha}\quad\hbox{ in $[T_1,T_2]\times (0,1/2]$.}
\ee

Fix some $t_1\in (0,T^*)$ and, for any $\eps\in (0,1/2)$, set $Q_\eps=(t_1,T_2]\times (\eps,1/2)$.
Denote $H=\partial_t-\partial_x^2$. 
By \eqref{SGBUprofileUpperEst2}, \eqref{SGBUprofileUpperEst3}, \rife{vhj1} and \rife{u_t-bounded},  
we have $|u_x|\le C\eps^{-\alpha}$ and $|u_{xx}|\le C\eps^{-\alpha-1}$ in $Q_\eps$, hence 
$|H(u_x)|=p|u_{xx}||u_x|^{p-1}\le C\eps^{-2-\alpha}$ in $Q_\eps$.
By standard 
parabolic estimates, for each $q\in (1,\infty)$, it follows that 
\be\label{bounduxt1aux3}
\|u_{xt}\|_{L^q(Q_\eps)}\le C_q\eps^{-2-\alpha}.
\ee
(Here and below, $C_q$ denotes a generic positive constant depending on $q$ but independent of $\eps$).
Next, since $|H(u_t)|=p|u_{xt}||u_x|^{p-1}$, we deduce from 
\rife{bounduxt1aux3}  that
$\|H(u_t)\|_{L^q(Q_\eps)}\le C_q\eps^{-3-\alpha}$. Since $|u_t|\le M$ in $Q_\eps$, parabolic estimates now give
$$\|u_{tt}\|_{L^q(Q_\eps)}+\|u_{txx}\|_{L^q(Q_\eps)}\le C_q\eps^{-3-\alpha}.$$
Taking $q$ large enough and using standard interpolation properties, we deduce that
$\|u_{tx}\|_{L^\infty(Q_\eps)}\le C_q\eps^{-3-\alpha}$, hence the claim.

\smallskip

{\bf Step 3.} {\it Comparison argument.}
By \rife{bounduxt1claim} and the mean value theorem, we have
\be\label{bounduxt1aux4}
|w(t,x)|\le C_0x^{-3-\alpha}\quad\hbox{ in $[T_1,T_2-h]\times (0,1/2]$,}
\ee
with $C_0>0$ independent of $h$.  
It then follows from \rife{bounduxt1aux2a}, \rife{bounduxt1aux2b} that
$w$ satisfies the following parabolic inequality with {\it strong nonlinear absorption}:
$$w_t-w_{xx}+C_1|w|^\gamma w\le B(t,x)w_x\quad\hbox{ in $[T_1,T_2-h]\times (0,1/2]$,}$$
where $\gamma=2/(3+\alpha)>0$ and $C_1>0$ is independent of $h$.  

To take care of the boundary conditions at $x=0$, we write
$$|(u_x)(t+h,x)-u_x(t,x)|\le |(u_x)(t+h,x)-U_*'(x)|+|(u_x)(t,x)-U_*'(x)|$$
and apply \rife{shiftedcopy} to obtain
\be\label{bounduxt1aux6}
|w(t,x)|\le 2Kh^{-1}x\quad\hbox{ in $[T_1,T_2-h]\times (0,1/2]$.}
\ee
Therefore $w$ extends to a continuous function $w\in C([T_1,T_2-h]\times [0,1/2])$ with $w(t,0)=0$.

Now, for any fixed $t_0\in [T_1,T_2-h)$, the function
$\overline w(t,x)=\overline w(t):=[C_1\gamma(t-t_0)]^{-1/\gamma}$ satisfies
$$\overline w_t- \overline w_{xx} +C_1 \overline w^{1+\gamma}\ge  B(t,x)\overline w_x
\quad\hbox{ in $(t_0,\infty)\times (0,1/2]$.}
$$
{If $t-t_0<\de$, then $\overline w\geq (C_1\gamma \de)^{-1/\gamma}$; hence, setting $\delta=(C_0^{-\gamma}/C_1\gamma)a^{(3+\alpha)\gamma}$} and using \rife{bounduxt1aux4}, we get
$$\overline w(t,a)\ge C_0a^{-3-\alpha}\ge w(t,a)\quad\hbox{ for all $t\in (t_0,\min(T_2-h,t_0+\delta))$.}$$
Since also $|w|\le 2Kh^{-1}$ in $[T_1,T_2-h]\times [0,1/2]$, owing to \rife{bounduxt1aux6}, whereas $\overline w(t)\to \infty$ as $t\to t_0$, we may
apply the comparison principle with $\pm\overline w$, and we obtain
$$|w(t,x)|\le [C_1\gamma(t-t_0)]^{-1/\gamma}\quad\hbox{ in $(t_0,\min(T_2-h,t_0+\delta))\times (0,a]$.}$$ 
Letting $h\to 0_+$, it follows that, for any $t_0\in [T_1,T_2)$, we have
$$|u_{xt}|\le [C_1\gamma(t-t_0)]^{-1/\gamma}\quad\hbox{ in $(t_0,\min(T_2,t_0+\delta))\times (0,a]$,}$$ 
hence in $(t_0,\min(T_2,t_0+\delta)]\times (0,1/2]$ by interior regularity of $u$.
This immediately  yields the desired conclusion. 
\end{proof}

\begin{proof}[Proof of Lemma~\ref{boundutt}]
  Assume $t_0\in (0,T^*)$ without loss of generality. 
We consider the approximating sequence of solutions $u_k$ of \rife{app-1d}, with   $F_k$ given by \rife{eqaaa}.  
The function $w:=u_{k,tt}$ solves  the equation  
$$
w_t - w_{xx}=F_k'(u_{k,x})w_x+F_k''(u_{k,x})(u_{k,xt})^2\ge F_k'(u_{k,x})w_x\,.
$$
Since $u_{k,tt}=0$ on the boundary and, by \rife{eqabA0}, $u_{k,tt}(t_0,\cdot)=u_{tt}(t_0,\cdot)$ for all $k$ sufficiently large,  
we have $u_{k,tt}\ge -C(t_0)$ for all $t>t_0$, 
and the conclusion follows by passing to the limit.
\end{proof}

\begin{proof}[Proof of Theorem~\ref{proppersist}(ii)]
We first claim that 
\be\label{defphiut}
\phi(t):=\lim_{x\to 0_+} u_t(t,x)
\hbox{ exists for all $t\in (T_1,T_2]$}
\ee
and
\be\label{defphiut2}
u_t(t,0) \hbox{ exists, with $u_t(t,0)=\phi(t)$, for all $t\in (T_1,T_2]$}.
\ee
Let $t\in (T_1,T_2]$. Writing
$$u_t(t,x)=u_t(t,1/2)-\int_x^{1/2} u_{tx}(t,y)\,dy, \quad 0<x<1/2,$$
and using the bound \rife{bounduxt1}, we obtain \rife{defphiut}, {with 
$$
\phi(t)= u_t(t,1/2)-\int_0^{1/2} u_{tx}(t,y)\,dy\,.
$$
Moreover, since $u$ is smooth in $(0,1)$ and $|u_{tx}|$ is dominated, the right-hand side is a continuous function of $t$, so $\phi$ is continuous and actually satisfies for all $x\in (0,1/2)$:}
\be\label{defphiut3}
u_t(t,x)=\phi(t)+\int_0^x u_{tx}(t,y)\,dy,\quad 0<x<1/2.
\ee
Therefore, for all $t,s\in (T_1,T_2]$, by\rife{u_t-bounded}, \rife{defphiut} and dominated convergence, we have
$$
u(s,0)-u(t,0)=\lim_{x\to 0} \bigl(u(s,x)-u(t,x)\bigr)=\lim_{x\to 0} \int_t^s u_t(\sigma,x)\,d\sigma=\int_t^s \phi(\sigma)\,d\sigma,$$
hence \rife{defphiut2} {follows from the continuity of $\phi$}.

\smallskip
We next claim that 
\be\label{defphiutclaim}
\quad\hbox{the restriction of $u_t$ to $(T_1,T_2]\times[0,1/2]$ is continuous.}
\ee
Let $\eps>0$ and $t\in [T_1+\eps,T_2]$.
For any $y\in (0,1/2)$, $s\in [T_1+\eps,T_2]$ and $x\in (0,y)$, by \rife{defphiut2}, \rife{defphiut3} and \rife{bounduxt1}, we have
$$
\begin{aligned}
|u_t(s,x)-u_t(t,0)| 
&\le |u_t(s,x)-u_t(s,y)|+|u_t(s,y)-u_t(t,y))|+|u_t(t,y)-u_t(t,0)| \\
&\le \int_x^y |u_{tx}(s,\xi)|\,d\xi +|u_t(s,y)-u_t(t,y)| + \int_0^y |u_{tx}(t,\xi)|\,d\xi \\
&\le |u_t(s,y)-u_t(t,y)|+C(\eps)y.
\end{aligned}
$$
Since $u_t\in C((0,\infty)\times (0,1/2])$, it follows that
$$\limsup_{(s,x)\to (t,0)}|u_t(s,x)-u_t(t,0)|  \le C(\eps)y.$$
The claim \rife{defphiutclaim} follows by letting $y\to 0$.
\smallskip

We then show that:
\be\label{defphiutclaim2a}
\quad\hbox{the restriction of the function $t\mapsto u(t,0)$ to $[T_1,T_2]$ is of class $C^1$.}
\ee
In view of \rife{defphiut2} and \rife{defphiutclaim}, it suffices to show that
\be\label{defphiutclaim2}
\lim_{t\to {T_1}_+} \phi(t) \hbox{ exists.}
\ee
Fixing any $t_0<T^*$, by Lemma~\ref{boundutt} we may choose $C>0$ such that, 
for all $x\in (0,1)$, the function $t\to u_t(t,x)+Ct$ is nondecreasing with respect to $t\ge t_0$.
Therefore, letting $x\to 0$ and using \rife{defphiut}, we infer that $\phi(t)+Ct$ is nondecreasing with respect to $t\in (T_1,T_2]$.
Since it is bounded as a consequence of \rife{u_t-bounded} and \rife{defphiut}, it has a finite limit as $t\to T_{1,+}$, 
hence \rife{defphiutclaim2} and \rife{defphiutclaim2a}.
\smallskip

Let us finally show that:
\be\label{defphiutclaim4}
\quad\hbox{the restriction $V$ of $u-U_*$ to $Q:=[T_1,T_2]\times [0,1/2]$ is differentiable.}
\ee
It suffices to verify the differentiability at points $(t,0)$ with $t\in[T_1,T_2]$.
Fix $t\in[T_1,T_2]$ and take any $h$ such that $t+h \in [T_1,T_2]$ and $x\in (0,1/2]$.
Since $V$ is continuous in $Q$ and recalling \rife{defphiut2}, we deduce from the mean value theorem that
$$
\begin{aligned}
V(t+h,x)-V(t,0)
&=V(t+h,x)-V(t+h,0)+V(t+h,0)-V(t,0)\\
&=V(t+h,x)-V(t+h,0)+u(t+h,0)-u(t,0)\\
&=xV_x(t+h,\theta x)+u(t+h,0)-u(t,0),
\end{aligned}
$$
for some $\theta\in (0,1)$ depending on $t, h, x$.
Since, by \rife{shiftedcopy},
\be\label{boundVx}
|V_x|\le Kx\quad\hbox{ in $[T_1,T_2]\times (0,1/2]$}
\ee
and recalling \rife{defphiutclaim2a}, this guarantees \rife{defphiutclaim4}.
Moreover, we get
\be\label{defphiutclaim5}
\quad\hbox{$V_t(t,0)=u_t(t,0)$ for all $t\in (T_1,T_2]$ and $V_x(t,0)=0$ for all $t\in [T_1,T_2]$.}
\ee

\smallskip
In view of \rife{boundVx} and \rife{defphiutclaim5}, the $C^1$ regularity of $V$ 
is now  a consequence of \rife{defphiutclaim4} and \rife{defphiutclaim}.
\end{proof}

The last result of this section, which will enable us to complete the proof of Theorem~\ref{nonmin0},
establishes the two-piece monotonicity of $u$ in the time interval $[T^*,T^r]$.
The monotonicity in the weak sense is proved by using zero-number for time translates of the solution.
The possibility of a plateau is then ruled out by using the zero-number of $u_t$, Hopf's lemma
and the regularizing barriers from the proof of Lemma~\ref{barrier}.

\begin{prop} \label{twopiece}
Under the assumptions of Proposition~\ref{recon}, there exists $T_m\in (T^*,T^r)$
such that $u(t,0)$ is    increasing  on $[T^*,T_m]$ and    decreasing  on $[T_m,T^r]$,
with moreover 
\be\label{utneg3}
u_t\le 0\quad\hbox{ in $[T_m,\infty)\times (0,1)$.}
\ee
\end{prop}

\begin{proof}
{\bf Step 1.} {\it Monotonicity in the weak sense.}
Fix $t_0\in (0,T^*)$. By Lemma~\ref{defz}(iii), there exists $\tau_0>0$ such that
\be\label{prelimz0D2}
Z(u(t+\tau)-u(t))\le 2\quad\hbox{ for all $\tau\in (0, \tau_0)$ and $t\ge t_0$.}
\ee
Consider $T_m:=\sup E$ where 
$$E=\{T>T^*; \, \hbox{ $u(t,0)$ is  nondecreasing on $[T^*,T]$}\}\in (T^*,T^r),$$ 
noting that $E\ne\emptyset$ owing to \eqref{estB2} and that $E\subset (T^*,T^r)$ because of \rife{ConclEta}. 
By definition there exists $\eps_j\to 0_+$ such that $u(T_m+\eps_j,0)<u(T_m,0)$.

Since $Z(u(T_m+\eps_j)-u(T_m))\le 2$ for large enough $j$ by \eqref{prelimz0D2} and $u(T_m+\eps_j,1/2)<u(T_m,1/2)$ by Lemma~\ref{basic-prop0}(ii),
it follows that $u(T_m+\eps_j)\le u(T_m)$ in $[0,1]$.
We may apply the comparison principle in Proposition~\ref{compP2} to the global viscosity 
 solutions $u$ and $\tilde u:=u(\cdot+\eps_j,\cdot)$
(indeed the assumption $\tilde u(0,\cdot)\equiv u(\eps_j,\cdot)\in X$ is satisfied for $j$ large enough).
Therefore, $u(t+\eps_j)\le u(t)$ in $[0,1]$  
for all $t\ge T_m$.
This implies \eqref{utneg3}. 

Since we know now that $u(s,x)\le u(t,x)$ for all $T_m\le t\le s\le T^r$ and all $x\in (0,1)$,
by letting $x\to 0$, we deduce that $u(t,0)$ is  nonincreasing on $[T_m,T^r ]$,
whereas  $u(t,0)$ is  nondecreasing on $[T^*,T_m]$ by the definition of $E$.

\smallskip

{\bf Step 2.} {\it Strict monotonicity.}
To conclude the proof, we shall show that $u$ cannot have a plateau.
Let us  assume for contradiction that $u(t,0)=\ell>0$ on $[T_1,T_2]$ for some $T^*<T_1<T_2<T^r$. 

\vskip2pt
We first claim that $N(t)=0$ on $[T_1,T_2]$.
Otherwise $u_t>0$ at some point of $[T_1,T_2]\times (0,1)$.
By continuity, it follows that there exist $a\in (0,1/2)$ and $T_1<t_1<t_2<T_2$, such that
$u_t(t,a)>0$ and $N(t)=2$ for all $t\in [t_1,t_2]$.   
Recalling \eqref{utcenter}, we deduce that 
$u_t\ge 0$ on $[t_1,t_2]\times (0,a]$.
By Remark~\ref{remb0}, we get $u_x\ge U_*'$ and $u\ge\ell+U_*$ on $[t_1,t_2]\times (0,a]$.
The function $w:=u-\ell-U_*$ is nonnegative, continuous up to the boundary, and verifies
\be\label{eqparabHopf}
w_t-w_{xx}=(u_x)^p-(U_*')^p\ge 0
\ee
in $Q:=(t_1,t_2)\times (0,a]$ with $w(t,a)>0$ (due to $u_t(t,a)>0$). 
Then, by comparing with the linear heat equation and using Hopf's Lemma, 
we deduce that $w(t,x)\ge c(t)x$ in $Q$ for some $c(t)>0$, i.e. 
$$u(t,x)\ge u(t,0)+U_*+c(t)x.$$
But this contradicts estimate \rife{shiftedcopy0}.

\vskip2pt
We have thus proved that $u_t\le 0$ on $[T_1,T_2]\times (0,1)$,
hence $u_x\le U_*'$ and $u\le \ell+U_*$ on $[T_1,T_2]\times (0,1/2]$  by Lemma~\ref{separation}(ii). 
The new function $w:=U_*-u+\ell$ is nonnegative and verifies~\rife{eqparabHopf} 
in $Q:=(T_1,T_2)\times (0,1/2]$. 
Since $w(T_1,\cdot)\not\equiv 0$  {(owing to $w_x(T_1,1/2)>0$)}, 
 Hopf's Lemma again gives $w(t,x)\ge d(t)x$ in $Q$ for some $d(t)>0$,
i.e. $u(t,x)\le u(t,0) +U_*-d(t)x$. 
But this again contradicts estimate \rife{shiftedcopy0}. 
\end{proof}

We have now all the ingredients to conclude the proof of Theorem~\ref{nonmin0}.
\vskip1em

\begin{proof}[Completion of proof of Theorem~\ref{nonmin0}.]
Assertion (i) follows from Proposition~\ref{nonmin} and the first part of Proposition~\ref{recon}.
\vskip2pt

In assertion (ii), estimates \rife{estA0} and \rife{BeqConclA0} are just \rife{estA} and \rife{BeqConclA} in Propositions~\ref{nonmin} and \ref{recon}
whereas \eqref{estB0} and \eqref{BeqConclAa0} follow from \eqref{estB}, \eqref{BeqConclAa} and
  the regularity of $u(t,0)$ proved in Theorem~\ref{proppersist}(ii). 
\vskip2pt

Assertion (iii) is a consequence of Theorem~\ref{proppersist} and Proposition~\ref{twopiece}.
\end{proof}

\section{   Behavior of minimal GBU solutions:  proof of Theorem~\ref{minimal0}} 
\label{Sec-min}

This section is devoted to the behavior of minimal blow-up solutions. We already know that  those solutions are characterized by the fact that the boundary condition is not lost. In addition, if $N(u_t(0))=2$, we also know from Proposition~\ref{nonmin}(i), that 
minimal blow-up solutions occur if and only if $L=0$, where $L$ is defined in \eqref{defL}.
  
We first give the:

\begin{proof}[Proof of Theorem~\ref{minimal0}(ii).] 
  By \eqref{utneg} and Proposition \ref{utneq}, it follows that
\be\label{utneg2}
u_t\le 0\quad\hbox{ in $[T^*,\infty)\times(0,1)$,}
\ee
and that $u$ never loses boundary conditions and satisfies $u\leq U_*$, $u_x\le U_*'$ for all times after $T^*$.
Now, set  $W:=U_*-u$. This function is nonnegative and verifies
$$
W_t-W_{xx}=(U_*')^p-(u_x)^p\ge 0
$$ 
on $(0,1/2)$ for $t>T^*$, with $W(t,1/2)>0$ due to \eqref{utcenter}   
(and $W$ is continuous up to the boundary).
Then,  Hopf's Lemma gives $W(T^*+t,x)\ge c(t)x$ for all $t>0$, with some (possibly bad) $c(t)>0$.
i.e. $u(T^*+t,x)\le U_*-c(t)x.$ We conclude by applying Lemma~\ref{barrier}.
\end{proof}

One of the main ingredients of the proof   of parts (i) and (iii) of Theorem~\ref{minimal0}   is the already established Proposition~\ref{contut},
where we showed that nonminimal blow-up and smoothing rates stem from the continuity of the time derivative at the boundary.
In the following steps, we aim at getting the necessary informations on the time derivative in order to apply Proposition~\ref{contut}.

\begin{lem}\label{minlem1}
Let $\phi\in W^{3,\infty}(0,1)$ be compatible at order two,
with $\phi$ symmetric and  nondecreasing on $[0,1/2]$.
Assume that $N(0)=2$  
and $T^*<\infty$.
Fix any $t_0\in (0,T^*)$. Then we have 
\be\label{boundutxA}
u_{tx}\ge -C,\quad (t,x)\in [t_0,T^*)\times [0,1/2],
\ee
for some constant $C>0$.
As a consequence, we have
\be\label{boundutxB}
u_t(t,x)\ge -Cx,\quad (t,x)\in [t_0,T^*]\times (0,1/2].
\ee
\end{lem}

\begin{proof}
Set $w:=u_{tx}$. Differentiating the equation in \rife{vhj1} with respect to $t$ and then to $x$, we find
$$
w_t-w_{xx}=p|u_x|^{p-2}u_x u_{xtx}+p(p-1)|u_x|^{p-2}u_{xx}u_{xt}
$$
that is, 
$$
w_t-w_{xx}=aw_x+bw
$$
with $a(t,x)=p|u_x|^{p-2}u_x$ and $b(t,x)=p(p-1)|u_x|^{p-2}u_{xx}$. 
Moreover, we observe that
$$
b=p(p-1)|u_x|^{p-2}(u_t-|u_x|^p)\le p(p-1)|u_x|^{p-2}(M-|u_x|^p)<C_0.
$$
Since the modified function $\tilde w=e^{-C_0t}w$ then satisfies
$$
\tilde w_t-\tilde w_{xx}=a\tilde w_x+b_1\tilde w,
$$
with $b_1:=b-C_0<0$, the maximum principle yields
\be\label{maxpplw}
\min_{\overline Q_T}\tilde w_-= \min_{\partial_PQ_T}\tilde w_-,
\quad\hbox{ for all $T\in (t_0,T^*)$,}
\ee
where $Q_T=(t_0,T)\times (0,1/2)$ and $\partial_PQ_T$ denotes its parabolic boundary.

On the other hand, 
  by Lemma~\ref{defz}(i),  
   for each given $t\in [t_0,T^*)$, we have 
$u_t(t,x)\ge 0$ for $x$ close to $0$, hence $u_{tx}(t,0)\ge 0$ (recalling $u_t(t,0)=0$).
Moreover, $\tilde w\ge -C$ at $x=1/2$,
since the solution remains smooth for all times away from the boundary of $(0,1)$.
Property \rife{maxpplw} then reduces to 
$$
\displaystyle\min_{\overline Q_T}\tilde w_- \ge  
\min\bigl[-C,\displaystyle\min_{x\in[0,1]}\tilde w_-(t_0,x)\bigr]
\quad\hbox{ for all $T\in (t_0,T^*)$,}
$$
and the conclusion follows.
\end{proof}

\begin{lem}\label{minlem2}
Under the assumptions of Theorem~\ref{minimal0}, we have
$$
\lim_{t\to T^*_-,\,x\to 0}u_t(t,x)=0.
$$
\end{lem}

\begin{proof}
We have $\displaystyle\liminf_{t\to T^*_-,\,x\to 0}u_t(t,x)\ge 0$ owing to \rife{boundutxB}.
So it suffices to show $\displaystyle\limsup_{t\to T^*_-,\,x\to 0}u_t(t,x)\le 0$.   

Assume for contradiction that there exist $\eta>0$, $t_j\to T^*_-$ and $x_j\to 0$ such that
$$u_t(t_j,x_j)\ge \eta.
$$
Let $h=\eta/(2C)$ where $C$ is given by \rife{boundutxA}.
For all large $j$, we have $x_j<h$ and we deduce from \rife{boundutxA} that
$$
u_t(t_j,h):=u_t(t_j,x_j)+\int_{x_j}^h u_{tx}(t_j,x)\, dx\ge \eta-Ch=\eta/2.
$$
Letting $j\to\infty$ and using $u_t\in C((0,T^*]\times (0,1))$, we deduce that
$u_t(T^*,h)\ge\eta/2$. 
This is a contradiction    with \eqref{utneg2}.  
\end{proof}

\begin{lem}\label{minlem3}
Under the assumptions of Theorem~\ref{minimal0}, we have
$$
\lim_{t\to T^*_+,\,x\to 0}u_t(t,x)=0.
$$
\end{lem}

\begin{proof} 
In view of \rife{utneg2},
it suffices to prove that $\displaystyle\liminf_{t\to T^*_+,\,x\to 0}u_t(t,x)\ge 0$.  
By \eqref{utcenter}, we have  
\be\label{BoundSigma2}
\sigma:=-\max_{t\in [T^*,T^*+1]} u_t(t,1/2)>0
\ee
and, by Lemma~\ref{defz}(iii), there exists $\eps_0\in (0,\sigma)$ such that
\be\label{prelimz0C2b}
Z_\eps(t):=Z(u_t(t)+\eps)\le 2\quad\hbox{  for all $\eps\in (0,\eps_0)$ and all $t\ge T^*$.}
\ee
Moreover, for all $t\in (T^*,T^*+1]$, since $u_t(t)\in C([0,1])$ and $u_t(t,0)=0$ by Theorem~\ref{minimal0}(ii), we have $Z_\eps(t)=2$ due to \eqref{BoundSigma2}.

Now assume for contradiction that there exist $\eps\in (0,\eps_0)$, $t_j\to T^*_+$ and $x_j\to 0$ such that
$$u_t(t_j,x_j)<-\eps.
$$
Then, we have
$$
u_t(t_j,x)\le - \eps, \quad x_j\le x\le 1/2
$$
due to $Z_\eps(t_j)=2$.
For all fixed $x\in (0,1/2)$, letting $j\to\infty$, we obtain
$u_t(T^*,x)\le - \eps$, a contradiction with \rife{boundutxB} for $t=T^*$.
\end{proof}

We have now all the ingredients to conclude the proof of Theorem~\ref{minimal0}.

\begin{proof}[Proof of Theorem~\ref{minimal0}(i) and (iii).]
Assertion (i) is a direct consequence of {Proposition~\ref{contut}(i) }and Lemma~\ref{minlem2},
whereas assertion (iii) is a direct consequence of {Proposition~\ref{contut}(ii)} and Lemma~\ref{minlem3}.
\end{proof}

\section{Appendix 1. A partial monotonicity result for the zero-number of $u_t$} 
\label{Sec-monot}

 In this short appendix, we state  the following additional result, announced before Proposition~\ref{zeronumberut2},
and prove it as an a posteriori consequence of our analysis of the behavior of the global viscosity solution.
It shows the monotonicity of the zero-number of $u_t$ for the viscosity solutions starting with $N(0)\le 4$.

\begin{prop} \label{zeronumbermonotone}
Let $\phi\in W^{3,\infty}(0,1)$ be compatible at order two,
with $\phi$ symmetric and  nondecreasing on $[0,1/2]$.
Assume that $N(0)\le 4$ and $T^*<\infty$.
Then the zero-number $N(t)$ of $u_t$ is a nonincreasing function of $t$.
In particular, if $N(0)=2$, then we have
$$N(t)=
\begin{cases}
2&\hbox{ for all $t\in [0,T_d)$}, \\
0&\hbox{ for all $t\in [T_d,\infty)$,}
\end{cases}
$$
with $T_d=T_m$ if  
$u$ is nonminimal and $T_d=T^*$ if $u$ is minimal.
\end{prop}

\begin{proof}
By Lemma~\ref{higherN}, the monotonicity property can be reduced to the case $N(0)=2$.
The case when $N(0)=2$ and $u$ is minimal follows from \rife{utneg2} and  Lemma~\ref{defz}.
Let us thus assume that $N(0)=2$ and $u$ is nonminimal.

By Lemma~\ref{defz}, \eqref{eqposut} and Proposition~\ref{twopiece}, we have $N(t)=2$ for all $t\in (0,T^*]$ and $N(t)=0$ for all $t\in [T_m,\infty]$.
Assume for contradiction that $N(t_0)=0$ for some $t_0\in (T^*,T_m)$.
Then $u_t(t_0)\le 0$ in $(0,1)$, hence 
\be\label{monotN1}
\quad\hbox{ $u(t_0,x)\le u(t_0,0)+U_*(x)$ in $[0,1]$}
\ee
by Lemma~\ref{separation}(ii) and the symmetry of $u$.  
Since $u(t,0)$ is  increasing on $[T^*,T^m]$ and $u_t\in C((T^*,T^r)\times [0,1])$ by Theorem~\ref{proppersist}(ii),
there exist $\eps>0$ and $t_1,t_2$ with $t_0<t_1<t_2<T_m$ such that
$$\hbox{\quad $u_t(t,0)\ge \eps$ on $[t_1,t_2]$ \quad and \quad $u_t(t,0)\ge 0$ on $(T^*,T_m)$.}$$
We first define the comparison function $z(t,x)=u(t,0)+U_*(x)$,
which satisfies
$$z_t-z_{xx}-|z_x|^p=u_t(t,0)-U_*''-U_*'^p=u_t(t,0)\ge 0\quad\hbox{ in $(T^*,T_m)\times (0,1)$}.$$
In view of \rife{monotN1} and of $u(t,1)=u(t,0)$, the comparison principle in Proposition~\ref{compP} guarantees that 
\be\label{monotN2}
u(t,x)\le u(t,0)+U_*(x) \quad\hbox{ in $(t_0,T_m)\times [0,1]$}.
\ee

We next define the comparison function
$$W(t,x)=u(t,0)+U_*(x)-\eta(t-t_1)x
\quad\hbox{ in $[t_1,t_2]\times [0,1]$},$$
where $\eta\in (0,\eps)$ is chosen small enough so that
$U_*(x)-\eta(t_2-t_1)x\ge 0$ and $U_*'(x)-\eta(t_2-t_1)\ge 0$ for all $x\in (0,1]$.
We then have $0\le W_x\le U_*'$ in $[t_1,t_2]\times (0,1]$ and $W(t,x)\ge u(t,0)$ in $[t_1,t_2]\times [0,1]$.
Consequently, we obtain
$$
\begin{aligned}
W_t-W_{xx}-|W_x|^p
&=u_t(t,0)-\eta x-U_*''-(W_x)^p \\
&\ge \eps-\eta-U_*''-U_*'^p=\eps-\eta>0\quad\hbox{ in $[t_1,t_2]\times (0,1)$.}
\end{aligned}
$$
Since $W(t_1,x)=u(t_1,0)+U_*(x)\ge u(t_1,x)$ in $[0,1]$ by \rife{monotN2},
applying the comparison principle again, we deduce that
$$u(t,x)\le u(t,0)+U_*(x)-\eta(t-t_1)x  
\quad\hbox{ in $[t_1,t_2]\times [0,1]$}.$$
But this contradicts  \rife{shiftedcopy0}.
\end{proof}

\section{Appendix 2: Proof of Theorem~\ref{prelimprop2}}  
\label{Sec-App}

We first establish the following interior a priori estimate for a more general equation, 
which covers the regularized problems \rife{approxpbm}.
It relies on a Bernstein type argument
(which is a modification of \cite[Theorem~3.1]{SZ06}, where only the case $F(\nabla u)=|\nabla u|^p$ was treated;
the latter was motivated by \cite{Li85} for the elliptic case). {See also the Appendix in \cite{PZ} for {\it local in time -- 
global in space} gradient estimates of similar type.}

 \begin{lem}\label{prelimprop2lem}
Let $\Omega$ be any domain of $\R^n$.
Let $F(\xi)\in W^{1,\infty}_{loc}(\R^n)$ satisfy:
\be\label{approxHyp4a}
F(\xi)\ge C_2|\xi|^m,\quad\hbox{ for $|\xi|\ge \xi_0$,}
\ee
\be\label{approxHyp3a}
|\nabla F(\xi)|\le C_1\bigl(1+|\xi|^{-2}F^{2-\theta}(\xi)\bigr), \quad\hbox{ for all $\xi\ne 0$,}
\ee
where $m>1$, $\theta\in (0,1), C_1, C_2, \xi_0>0$.
Let $0\le t_0<T<\infty$, $v\in C^{1,2}((t_0,T)\times\Omega)$, with $\nabla v\in C([t_0,T)\times\Omega)$, be a solution of 
$$v_t-\Delta v =F(\nabla v) \quad\hbox{ in $(t_0,T)\times\Omega$},$$
and assume that
\be\label{approxHyp4a0}
|\nabla v(t_0,x)|\le M_0 \quad\hbox{ in $\Omega$}
\ee 
and
\be\label{approxHyp4a1}
|v_t|\le M_1 \quad\hbox{ in $(t_0,T)\times\Omega$}.
\ee 
Then $\nabla v$ satisfies the estimate
\be\label{approxpbm2aEstim}
|\nabla v(t,x)|\le C_3\bigl[1+\delta^{-1/(\theta m)}(x)+\delta^{-1/(m-1)}(x)\bigr] \quad\hbox{ in $(t_0,T)\times\Omega$},
\ee 
where $\delta(x)={\rm dist}(x,\partial\Omega)$ and the constant $C_3>0$ depends only on $C_1,C_2,\xi_0,M_0,M_1,m,\theta,n$.
\end{lem}

\begin{proof}
Let $x_0\in\Omega$, $R=\delta(x_0)$ and $Q_R=(t_0,T)\times B_R(x_0)$.
For $i=1,\dots,n$, put $v_i={\partial v\over\partial x_i}$. 
By parabolic regularity, we have $\partial_tv_i, D^2v_i\in L^q_{loc}(Q_R)$
for all $q<\infty$.
Differentiating in $x_i$, we obtain
\begin{equation}  \label{eqLocBernsteinPhA}
\partial_t v_i-\Delta v_i= (F(\nabla v))_i.
\end{equation}
Let $w=|\nabla v|^2$. Multiplying (\ref{eqLocBernsteinPhA}) by $2 v_i$, summing up, and using
$\Delta w=2\nabla v\cdot\nabla(\Delta v)+2|D^2v|^2$ (with
$|D^2v|^2=\textstyle\sum_{i,j}(v_{ij})^2$),
we deduce that
\begin{equation}  \label{eqLocBernsteinPhB}
{\cL}w:= w_t-\Delta w+b\cdot\nabla w=-2|D^2v|^2 \ \hbox{ in $Q_R$,\quad
with } b:= -\nabla F(\nabla v).
\ee
Let $a\in (0,1)$ and put $R'=3R/4$. We can select a cut-off function $\eta\in C^2(\overline
B(x_0,R'))$, $0\leq \eta\leq 1$, with $\eta=0$
for $|x-x_0|=R'$, such that
\begin{equation}  \label{condCutoff}
|\nabla\eta| \leq C R^{-1}\eta^a,\quad |\Delta\eta|+\eta^{-1}|\nabla\eta|^2\leq C R^{-2}\eta^a,
\quad\hbox{ for $|x-x_0|<R'$,}
\end{equation}
with $C=C(a)>0$ (see e.g. \cite[p. 374]{SZ06}).
Set
$$z=\eta w.$$
In the rest of the proof, $C$ denotes a generic positive constant depending only on $C_1,C_2,\xi_0$, $M_0,M_1,m,\theta,n,a$.
We have
$$
{\cL}z =\eta{\cL}w+w{\cL}\eta-2\nabla\eta\cdot\nabla w
=-2\eta|D^2v|^2+w{\cL}\eta-2\nabla\eta\cdot\nabla w
\quad\hbox{ in $Q_{R'}$}
$$
by (\ref{eqLocBernsteinPhB}), and
$$
2|\nabla\eta\cdot\nabla w|=4\big|\sum_i\nabla\eta\cdot v_i\nabla
v_i\big| \leq 4\sum_i \eta^{-1}|\nabla\eta|^2 v_i^2+\sum_i \eta|\nabla
v_i|^2 =4\eta^{-1}|\nabla\eta|^2 w+\eta|D^2v|^2,
$$
hence, using \rife{condCutoff},
$$ 
{\cL}z+\eta|D^2v|^2
\leq ({\cL}\eta+4\eta^{-1}|\nabla\eta|^2)w
\leq C\eta^a \bigl(R^{-1} |\nabla F(\nabla v)|+R^{-2}\bigr)|\nabla v|^2
\quad\hbox{ in $Q_{R'}.$}
$$
Taking \rife{approxHyp3a} into account and denoting $\Sigma=Q_{R'}\cap\{(t,x);\,z(t,x)\ge \xi_0^2\}$, we obtain
$$
{\cL}z+\eta|D^2v|^2
\leq CR^{-1}\eta^a\,F^{2-\theta}(\nabla v)+C(R^{-2}+R^{-1})\eta^a |\nabla v|^2
\quad\hbox{ in $\Sigma.$}
$$
Choose $a=\max\bigl((2-\theta)/2,1/m\bigr)\in (0,1)$. Using Young's inequality, $\eta\le 1$ and \rife{approxHyp4a}, it follows that,
for all $\eps>0$,
\begin{align}
{\cL}z+\eta|D^2v|^2
&\leq \eps\eta^{2a/(2-\theta)} F^2(\nabla v)+C(\eps)R^{-2/\theta}
+\eps\eta^{am}|\nabla v|^{2m}+C(\eps)(1+R^{-2m/(m-1)}) \notag \\
&\leq \eps\eta(1+C_2^{-2}) F^2(\nabla v)+C(\eps)(1+R^{-2/\theta}+R^{-2m/(m-1)})
\quad\hbox{ in $\Sigma.$}
  \label{eqLocBernsteinPh2}
\end{align}
On the other hand, using the fact that $|F(\nabla v)-v_t|=|\Delta v|\leq \sqrt{n}|D^2v|$, along with \rife{approxHyp4a1}, we obtain
$$
{1\over 2n}F^2(\nabla v)\leq |D^2v|^2+|v_t|^2\leq |D^2v|^2+M_1^2.
$$
Combining this with (\ref{eqLocBernsteinPh2}) with the choice $\eps=[4n(1+C_2^{-2})]^{-1}$, it follows that
\begin{equation}  \label{eqLocBernsteinPh2b}
{\cL}z+{1\over 2n}\eta F^2(\nabla v)
\leq {1\over 4n}\eta F^2(\nabla v)+C(1+R^{-2/\theta}+R^{-2m/(m-1)})
\quad\hbox{ in $\Sigma.$}
\end{equation} 
Since $\eta F^2(\nabla v)\ge \eta |\nabla v|^{2m} \ge z^m$ in $\Sigma$, we obtain
$${\cL}z\leq -{1\over 4n}z^m+A  \quad\hbox{ in $\Sigma$,
\quad with $A=C(1+R^{-2/\theta}+R^{-2m/(m-1)})$.}$$
It follows from the maximum principle (see, e.g., \cite[Proposition~52.4 and Remark~52.11(a)]{QS07}) that
$$\sup_{Q_{R'}} z\leq \max\bigl(\max_{x\in\overline{B_{R'}}} z(t_0,x),(4nA)^{1/m}\bigr)
\leq \max\bigl(M_0^2,(4nA)^{1/m}\bigr),$$ 
hence
$$|\nabla u(t,x_0)|\leq \sup_{Q_{R'}} z^{1/2}\leq
C\bigl[1+R^{-1/(\theta m)}+R^{-1/(m-1)}\bigr], \quad t_0<t<T,$$
which proves the lemma.
\end{proof}

\begin{proof}[Proof of Theorem~\ref{prelimprop2}(i).]
{\bf Step 1.} {\it Convergence.}
For each $k$, since $F_k$ has at most quadratic growth by \rife{approxHyp4b}, 
the global classical existence of $u_k$ is well known (see, e.g., \cite[Section 35]{QS07}).
In view of~\rife{approxHyp1}, the maximum principle guarantees that $u_k$ is nondecreasing with respect to $k$ and that 
\be  \label{BernsteinPhAux0a}
e^{t\Delta}\phi\le u_k\le A_0:= \sup_\Omega\phi \quad\hbox{ in $[0,\infty)\times\Omega$,}
\ee
for all $k\ge 1$. Therefore $U(t,x):=\lim_k u_k(t,x)$ is well defined and finite for each $t>0$, $x\in\Omega$.
Moreover, by \rife{approxHyp1}, \rife{approxHyp3}, we have
\be  \label{BernsteinPhAux0}
F_k(\xi)\le |\xi|^p, \qquad |\nabla F_k(\xi)|\le C(|\xi|^q+1),
\ee
with $q=(2-\theta)p-2>0$ and $C$ independent of $k$. 
By standard parabolic regularity arguments, we deduce the existence  
of $t_0>0$ such that, for each $\eps\in (0,t_0)$,
\be  \label{BernsteinPhAux1}
C_0:=\sup_k \|\nabla u_k\|_{L^\infty([0,t_0]\times\overline\Omega)}<\infty\quad\hbox{ and }\quad
\sup_k \|u_k\|_{C^{1,2}([\eps,t_0]\times\overline\Omega)}<\infty.
\ee
In particular, $M_0:=\sup_k \|\partial_t u_k\|_{L^\infty(\Omega)}<\infty$ and,
by the maximum principle applied to $\partial_t u_k$, we deduce that
\be  \label{BernsteinPhAux2}
|\partial_t u_k|\le M_0\quad\hbox{ in $[t_0,\infty)\times\Omega$.}
\ee

On the other hand, the nonlinearities $F=F_k$ satisfy the assumptions of Lemma~\ref{prelimprop2lem} 
with $m=2$ and, by \rife{BernsteinPhAux1}, \rife{BernsteinPhAux2},
the solutions $v=u_k$ satisfy assumptions \rife{approxHyp4a0}, \rife{approxHyp4a1}, 
with constants $C_1,C_2,\xi_0,M_0,M_1$ independent of $k$.
Applying Lemma~\ref{prelimprop2lem}, we obtain an interior a priori estimate of the form
\be  \label{BernsteinPhAux3}
|\nabla u_k|\le C(\eps)\quad\hbox{ in $[t_0,\infty)\times\Omega_\eps$},
\ee
for each $\eps>0$, where $\Omega_\eps=\{x\in\Omega;\, \delta(x)>\eps\}$. 
By  \rife{BernsteinPhAux1}, \rife{BernsteinPhAux3}, 
parabolic estimates and a diagonal procedure, it follows that the sequence $u_k$ is 
relatively compact in $C^{1,2}((\eps,T)\times\Omega_\eps)$ for each $\eps, T>0$.
Consequently, using \rife{approxHyp1}, we deduce that $U\in C^{1,2}((0,\infty)\times\Omega)$, 
that some subsequence of $u_k$ converges to $U$ in $C^{1,2}_{loc}((0,\infty)\times\Omega)$
and that $U$ is a classical solution of $U_t-\Delta U=|\nabla U|^p$ in $(0,\infty)\times\Omega$,
which also satisfies 
\be  \label{BernsteinPhAux6}
U\in C^{1,2}((0,t_0]\times\overline\Omega),
\ee
 owing to \rife{BernsteinPhAux1}. 
Moreover, the whole sequence $u_k$ actually converges due to the uniqueness of the possible (pointwise) limits.

\smallskip
{\bf Step 2.} {\it Continuity of $U$ up to the boundary.}
By \rife{BernsteinPhAux2}, we have
\be  \label{BernsteinPhAux4c}
|U_t|\le M_0 \quad\hbox{ in $[t_0,\infty)\times\Omega$.}
\ee
Hence, for $t\geq t_0$,  $U(t)$ satisfies $\Delta U + |\nabla U|^p \leq M_0$ in $\Omega$. By 
\cite[Theorem 1.1]{CLP}, $U(t)$ can be extended in a continuous way up to  the boundary and belongs to $C^{1-\alpha}(\overline \Omega)$, $\alpha=1/(p-1)$, with  a uniform (in time) estimate
\be\label{hold}
|U(t,x)-U(t,y)|\leq M \, |x-y|^{1-\alpha},\qquad  t\geq t_0,\ x,y\in \overline \Omega,
\ee
where $M$ only depends on $M_0, \|U\|_\infty, p, \Omega$. On the other hand, as a consequence of \rife{BernsteinPhAux4c} and since $u$ is smooth inside $\Omega$, we also have
\be  \label{lip-time}
\|U(\tau)-U(t)\|_{L^\infty(\Omega)}\leq M_0 |\tau-t|\,,\quad \tau, t\in [t_0, \infty).
\ee
From \rife{hold} and \rife{lip-time}, one immediately deduces the continuity of $U$ up to the boundary for $t\geq t_0$, and therefore, since $U$ is classical in $[0,t_0]$, we have that $U\in C([0,\infty)\times\overline\Omega)$.

Moreover, since $U\ge e^{t\Delta}\phi$ in $[0,\infty)\times\Omega$ owing to \rife{BernsteinPhAux0a}, we also have
 \be  \label{BernsteinPhAux6c}
U\ge 0 \ \hbox{ on $(0,\infty)\times\partial\Omega$.}
\ee

\smallskip
{\bf Step 3.} {\it Uniqueness.}
We claim that $U$ is independent of the choice of the sequence $F_k$ verifying the assumptions of Theorem~\ref{prelimprop2}(i).
Let $V$ be the limit of the sequence $v_k$ associated with another sequence of nonlinearities $G_k$.
Since $\partial_t v_k-\Delta v_k=G_k(\nabla v_k)\le |\nabla v_k|^p$ and $v_k=0\le U$ on $\partial\Omega$ by~\rife{BernsteinPhAux6c},
the comparison principle in Proposition~\ref{compP} ensures that $v_k\le U$, hence $V\le U$.
Similarly, we obtain $U\le V$.
\end{proof}

\begin{proof}[Proof of Theorem~\ref{prelimprop2}(ii).]
Since $u_k=0\le u$ on $\partial\Omega$ and $u_k, u\in C([0,\infty)\times\overline\Omega)\cap C^{1,2}((0,\infty)\times\Omega)$,
the comparison principle in Proposition~\ref{compP} 
ensures that $u_k\le u$, hence $U\le u$.
On the other hand, $U\in C([0,\infty)\times\overline\Omega)\cap C^{1,2}((0,\infty)\times\Omega)$ is a supersolution of \rife{VHJ} in the classical sense, hence a viscosity supersolution.
By the comparison principle for viscosity sub-/supersolutions (see \cite{BDL04}), we conclude that $U\ge u$.
\end{proof}

{\bf Acknowledgements.}
Part of this work was done during visits of Ph.S. at the Dipartimento di Matematica of Universit\`a di Roma Tor Vergata. He wishes to thank this institution for the kind hospitality.
Ph.S is partially supported by the Labex MME-DII (ANR11-LBX-0023-01). A.P. was partially supported by the Grant \lq\lq Consolidate the Foundations 2015 (IrDyCo)\rq\rq\ of University of Rome Tor Vergata and by INDAM - GNAMPA funds (2018).

\end{document}